\documentclass[a4paper]{amsart}

\usepackage{subcaption}
\usepackage{capt-of}% or
\usepackage{booktabs}
\usepackage{varwidth}

\usepackage[export]{adjustbox}
\usepackage{tikz}
\DeclareRobustCommand\full  {\tikz[baseline=-0.6ex]\draw[thick] (0,0)--(0.5,0);}
\DeclareRobustCommand\dotted{\tikz[baseline=-0.6ex]\draw[thick,dotted] (0,0)--(0.54,0);}
\DeclareRobustCommand\dashed{\tikz[baseline=-0.6ex]\draw[thick,dashed] (0,0)--(0.54,0);}

\usepackage{tikz}
\newlength\replength

\usepackage{bbm}
\usepackage[utf8]{inputenc}  
\usepackage{amssymb}
\usepackage{enumerate}
\usepackage[shortlabels]{enumitem}
\usepackage{mathtools}
\usepackage{graphicx}
\usepackage{url}

\usepackage{color}  
\definecolor{darkred}{rgb}{0.6,0,0}
\definecolor{darkgreen}{rgb}{0,0.5,0}
\definecolor{darkmagenta}{rgb}{0.5,0,0.5}
\usepackage[pdftex, colorlinks=true,
linkcolor=blue,
citecolor=darkgreen,
urlcolor=darkmagenta]{hyperref}

\makeatletter
\def\@cite#1#2{\textup{[{#1\if@tempswa , #2\fi}]}}
\makeatother

\DeclarePairedDelimiter\abs{\lvert}{\rvert}

\DeclarePairedDelimiterX\Set[1]\{\}{%
	
	#1
}

\newcommand{\R}{\mathbb R}
\newcommand{\N}{\mathbb N}

\newcommand{\dott}{\, \cdot\,}

\usepackage[sort]{cite}
\usepackage{bigints}
\newcommand{\Z}{\mathbb{Z}}
\DeclareMathOperator{\lip}{Lip}
\DeclareMathOperator{\bv}{BV}
\newcommand{\sgn}{\mathop\mathrm{sgn}}
\newcommand{\norma}[1]{{\left\|#1\right\|}}

\renewcommand{\d}[1]{\mathinner{\mathrm{d}{#1}}}
\DeclareMathOperator{\TV}{TV}
\DeclareMathOperator{\lipR}{Lip(\R)}
\newcommand{\D}{\Delta}
\newcommand{\modulo}[1]{{\left|#1\right|}}
\renewcommand{\L}[1]{\mathbf{L^#1}}

\newtheorem{theorem}{Theorem}[section]

\newtheorem{lemma}[theorem]{Lemma}
\newtheorem{definition}[theorem]{Definition}

\newtheorem{remark}[theorem]{Remark}
\newtheorem{corollary}[theorem]{Corollary}

\numberwithin{equation}{section}     
\allowdisplaybreaks
\usepackage{graphicx} 
\usepackage{tabularx} 
\usepackage{multirow}
\usepackage{booktabs}
\begin{document}
	\title[Well-posedness and error estimates for nonlocal conservation laws]{Well-posedness and Error estimates  for Coupled systems of nonlocal conservation laws}
	\author[A. Aggarwal]{Aekta Aggarwal}
	\address[Aekta Aggarwal]{\newline
		Operations Management and Quantitative Techniques, Indian Institute of Management, Prabandh Shikhar, Rau--Pithampur Road, Indore, Madhya Pradesh, 453556, India}
	\email[]{\href{aektaaggarwal@iimidr.ac.in}{aektaaggarwal@iimidr.ac.in}} 
	
	\author[H. Holden]{Helge Holden}
	\author[G. Vaidya]{Ganesh Vaidya}
	\address[Helge Holden, Ganesh Vaidya]{\newline
		Department of Mathematical Sciences,
		NTNU  Norwegian University of Science and Technology,
		NO--7491 Trondheim, Norway}
	\email{\href{helge.holden@ntnu.no,ganesh.k.vaidya@ntnu.no}{helge.holden@ntnu.no,ganesh.k.vaidya@ntnu.no}}
	\urladdr{\href{https://www.ntnu.edu/employees/helge.holden,https://www.ntnu.edu/employees/ganesh.k.vaidya}{https://www.ntnu.edu/employees/helge.holden}, \href{https://www.ntnu.edu/employees/ganesh.k.vaidya}{https://www.ntnu.edu/employees/ganesh.k.vaidya}}

	\date{\today} 
	
	
	
	\thanks{The work was carried out during GV's tenure of the ERCIM ‘Alain Bensoussan’ Fellowship Programme, and AA's research visit at NTNU. It was supported in part by the project \textit{IMod --- Partial differential equations, statistics and data:
			An interdisciplinary approach to data-based modelling}, project number 325114, from the Research Council of Norway, and by AA's Faculty Development Allowance, funded by IIM Indore.}
	
	%\dedicatory{}
	
	\begin{abstract}
		This article deals with the error estimates for numerical approximations of the entropy solutions of coupled systems of nonlocal hyperbolic conservation laws.
		%with  local part of the  flux.
		%{, being possibly rough in the space variable. 
			The systems can be strongly coupled through the nonlocal coefficient present in the convection term. A fairly general class of fluxes is being considered, where the local part of the flux can be discontinuous at infinitely many points, with possible accumulation points. The aims of the paper are threefold: 1. Establishing existence of entropy solutions with rough local flux for such systems, by deriving a uniform $\bv$ bound on the numerical approximations; 2. Deriving a general Kuznetsov-type lemma (and hence  uniqueness) for such systems with both smooth and rough local fluxes;  3. Proving the convergence rate of the finite volume approximations to the entropy
			solutions of the system as $1/2$ and $1/3$, with homogeneous (in any dimension) and rough local parts (in one dimension), respectively. Numerical experiments are included to illustrate the convergence rates.
			
		\end{abstract}
		
		\maketitle
		
		%---------------- SECTION
		\section{Introduction}
		We here study the initial-value problem (IVP) for the coupled system of nonlocal hyperbolic conservation laws with discontinuous flux given by
		\begin{align}
			\label{eq:umulA}
			\partial_t \boldsymbol{U} + \partial_x \Big(\boldsymbol{F}(\boldsymbol{\sigma}(x)\circ\boldsymbol{U}))\circ \boldsymbol{\nu}(\boldsymbol{\mu} \circledast \boldsymbol{U})\Big) &= \boldsymbol{0},
			\,\quad (t,x) \in (0,T)\times\R, \\  \notag
			\boldsymbol{U}(0,x)&=\boldsymbol{U}_0(x), \quad x \in\R.
		\end{align} 
		Here the unknown is $\boldsymbol{U}=(U^1,\ldots,U^N): [0,\infty)\times \R \rightarrow \R^N$,  with $\circ$ denoting the Hadamard (component-wise) product. Furthermore, $\boldsymbol{\sigma}=( \sigma^1,\ldots, \sigma^N)\colon\R \rightarrow \R^N$, where each component
		$\sigma^k$ is allowed to have an infinitely many discontinuities, including accumulation points. The system is coupled due
		to the nonlocal
		terms  $\boldsymbol{\mu} \circledast  \boldsymbol{U}:\R \rightarrow \R^{N^2}$, with $\boldsymbol{\mu}$ a smooth $N\times N$ matrix, and $(\boldsymbol{\mu}\circledast\boldsymbol{U})^k:= \left(\mu^{k,1}*U^1,\ldots,\mu^{k,N}*U^N\right)$ with convolution $*$ over space only.  We have 
		$\boldsymbol{\boldsymbol{\nu}}=(\nu^1,\ldots,\nu^N)$, with each $\nu^k:\R^N \rightarrow\R$.
		The function $\boldsymbol{F}=(f^1,\ldots,f^N)\colon\R \rightarrow \R^N$ is Lipschitz in each component, with $\boldsymbol{F}(\boldsymbol{0})=\boldsymbol{0}$ and $\boldsymbol{\nu}(\boldsymbol{0})=0$, and $\abs{\boldsymbol{\sigma}}$ is strictly bounded away from zero.  Precise assumptions are given below.
		
		Hyperbolic conservation laws are partial differential equations (PDEs) that describe the conservation of $\boldsymbol{U}$ in a physical system, where the temporal rate of change of $\boldsymbol{U}$  is exactly balanced by the spatial rate of change of the flux function. The scalar case, i.e., with $N=1$, is well understood, also in multidimensions, while the case of systems, i.e., $N>1$, is considerably harder even in one dimension.  The use of Hadamard products restricts the coupling in \eqref{eq:umulA} to the convolution term $\boldsymbol{\nu}(\boldsymbol{\mu} \circledast \boldsymbol{U})$, allowing us to use the scalar theory. Hyperbolic conservation laws were originally introduced in the context of fluid dynamics, however, recently these equations have been applied as models covering a wide range of phenomena, and this has necessitated a considerably wider class of flux functions. Here we include two 
		effects---we let the flux function depend on space in a discontinuous manner through the function $\boldsymbol{\sigma}(x)$, and, furthermore, we let the flux function depend on a local average of  (the full vector) $\boldsymbol{U}$ in the form of the convolution term $\boldsymbol{\nu}(\boldsymbol{\mu} \circledast \boldsymbol{U})$.
		
		Some of the recent applications include crowds~\cite{CGL2012,CL2011,CHM2011}, traffic~\cite{BG2016,
			CHM2011}, opinion formation~\cite{ANT2007}, sedimentation
		models~\cite{BBKT2011}, structured population dynamics~\cite{Per2007}, supply chain models~\cite{CHM2011}, granular material dynamics~\cite{AS2012} and  conveyor belt dynamics~\cite{GHS+2014}. 
		% \\***PLEASE CHECK THIS**\\
		The case of smooth $\boldsymbol{F}$ has been well studied in literature, a non-exhaustive list being \cite{ACT2015, ACG2015,AG2016,BG2016,FGKP2022,CGL2012,CL2011,CHM2011}, and various references therein.

		Let us explain why it is necessary to study the extensions given in this paper. As we know, roads may experience rapidly changing local conditions, and when we model traffic flow by hyperbolic conservation laws, we allow for this by having a non-smooth, spatially depending function 
		$\boldsymbol{\sigma}(x)$ in the flux function.  Furthermore, more practical problems like the multi-agent systems, as well as laser technology, dynamics of multiple crowds/cars, lane formation, sedimentation and conveyor belts, see, for example, \cite{CMR2016, GKLW2016,CM2015}, one needs to analyse coupled systems of nonlocal conservation laws with discontinuous flux.
		
		The literature of the following \textit{local} counterpart of \eqref{eq:umulA} with $N=1$, viz. 
		\begin{equation*}
			U_t+(F(x,U)\nu(U))_x=0,
		\end{equation*} 
		is quite rich, with the results spanning to $F$ being a nonlinear function of $U$ and having general spatial discontinuities; see, for example, \cite{AMV2005,AJV2004,AMV2007, AKR2010,Per2007,GTV2022a,GTV2022,AVV2021,ASSV2020,GJT2020} and references therein. 
		
		However, the nonlocal counterpart \eqref{eq:umulA}, remains far from being settled, with the only exceptions being, \cite{AV2023}, \cite{FCV2023} and \cite{KP2021}. References \cite{KP2021,FCV2023} settle the problem for the case $F(x,U)= s(x)U$, with \cite{KP2021} establishing this via a fixed-point argument with a bounded discontinuous $s$ and monotone nonsymmetric kernels, and \cite{FCV2023} settling the same via a convergent numerical scheme for monotonically decreasing $s$, with a single spatial discontinuity and $\mu$ being a downstream kernel. The question of well-posedness for the cases of $F(x,U)= s(x)U$ with $ s$ having non discrete discontinuities, and for the case of $F(x,U)= s(x)f(U)$ with a nonlinear $f$ and $ s$ having single or finitely many spatial discontinuities, with or without any constraints on monotonicity of $\nu$ or $\mu$, remained unexplored and unsettled for a long time. This question was recently settled in \cite{AV2023}, where the well-posedness was explored in a very general setup with $F(x,U)=f( s(x)U),$ where $f$ is nonlinear and there are no additional constraints on $\nu$, $\mu$ and the discreteness or monotonicity of spatial discontinuities.
		
		The existence of the entropy solutions for \eqref{eq:umulA} for $N\ge 1$ with sufficiently smooth $\boldsymbol{\sigma}$ has been studied in \cite{ACG2015,BBKT2011,GHS+2014,CG2023,CM2015}, while the uniqueness has been settled in \cite{CM2015,CR2018,CMR2016}, for example. The question of existence as well as uniqueness of the entropy solutions for the system  \eqref{eq:umulA} with $\boldsymbol{\sigma}$ having any kind of spatial discontinuities, remains unexplored and unsettled as of now. This article settles these questions for initial-value problems \eqref{eq:umulA} which we can write on component form as
		\begin{align}
			\label{eq:uA}
			\partial_t U^{k} +\partial_x \Big(f^k( \sigma^k(x)U^k) \nu^k((\boldsymbol{\mu} * \boldsymbol{U})^k)\Big) &=0, \quad (t,x) \in Q_T:=(0,T)\times \R,
			\\ \label{eq:u11A}
			U^{k}(0,x)&=U_0^{k}(x), \quad x \in \R,
		\end{align} 
		for each $k=1,2,\ldots,N$.  Here $T$ is a given positive time.  The assumptions that $\sigma^k\in \bv(\R,\R^+)$ and $\inf_{x\in\R}\sigma^k(x)>0$
		are justified, being the essential technical assumptions required for well-posedness of scalar \textit{local} counterpart of \eqref{eq:uA}--\eqref{eq:u11A}, see \cite{GTV2022a,GTV2022,Pan2009,Tow2020} and references therein. As stated earlier, the system \eqref{eq:uA}--\eqref{eq:u11A} has been studied in \cite{KP2021} only for the linear case $ f^k(u)=u,$ using a fixed point argument. However, in the setup of this paper, $f^k$ is nonlinear and consequently, there can be multiple weak solutions, like in the local case. Hence, an entropy condition is required to single out the unique solution.
		
		In this article, this question will be settled in Section \ref{uni},  where it will be shown that any two  \textit{entropy} solutions (cf.~\eqref{kruz2}) of \eqref{eq:uA} corresponding to the same initial data \eqref{eq:u11A}, are in fact equal, and the approximate solutions obtained from the finite volume schemes proposed in Section \ref{sec:existence}, satisfy the discrete form of the entropy inequality \eqref{kruz2}. As a result, one will obtain that the approximate solutions converge to the unique entropy solution of \eqref{eq:uA}--\eqref{eq:u11A}. This uniqueness result is novel and is not dealt with in existing uniqueness results of \cite{BKRT2004,BG2016,KP2021,CR2018,CMR2016}, which deal with either local part of the flux being linear in state variable and possibly discontinuous in space variable, or non-linear smooth \textit{local} fluxes only. At the time of writing of this article, this is the first result establishing the existence and uniqueness of solutions of such coupled systems with discontinuous flux. It is to be noted that no technical restrictions on monotonicity of kernel or $\boldsymbol{\sigma}$ have been made in the analysis, and hence the article presents the first result in a very general setting. However, to exercise this method, we need a uniform $\bv$ bound on the entropy solution $\boldsymbol{U}$ and its numerical approximation $\boldsymbol{U}_{\Delta}$, which is one of the objectives of this article.
		
		The final objective of this article is the error analysis of our finite volume approximations, proposed in Section \ref{sec:existence}. To the best of our knowledge, this is the first result towards  error analysis of numerical approximations for coupled systems of nonlocal conservation laws, including the case of smooth $\sigma$. However, the literature on error estimates for local conservation laws is quite rich, see\cite{CS2012,BR2020,KR2001,GTV2022, KAR1994,SAB1997,KUZ1976} for $N=1$. This is achieved by following the techniques of~\cite{KUZ1976}, where the  so-called doubling of  variables argument is exploited to express the error between the entropy solution $U=u$ (i.e., $N = 1$) and its numerical approximation $u_{\Delta}$ in terms of their relative entropy. A first result towards scalar nonlocal conservation laws, i.e., with $N=1$, with a homogeneous smooth local part of the flux, was recently done in \cite{AHV2023}. In general, for any $N>1$, the local counterpart of \eqref{eq:uA} is fairly less explored with a limited list being \cite{BS2020,AMW2021}. This paper establishes  error estimates for $N>1$ for the nonlocal counterpart \eqref{eq:uA} with smooth local fluxes in any dimension, and with a fairly large class of rough local fluxes (precise assumptions described later) in one dimension.
		
		Thus, the novelties of this article for the system of nonlocal conservation laws \eqref{eq:umulA}, where $k^{th}$ equation is given by \eqref{eq:uA}, are as follows:
		\begin{itemize}
			\item Existence of the \textit{adapted entropy} solutions via finite volume approximations.
			\item A Kuznetsov-type lemma and thereby the uniqueness of the entropy solution.
			\item Convergence rate estimates for the proposed numerical schemes.
		\end{itemize}
		
		The paper is organized as follows. In Section~\ref{def}, we introduce some definitions and notation to be used in the article. In Section~\ref{sec:existence}, we present the convergence of the Godunov type and the Lax--Friedrichs type scheme of the approximation of the IVP \eqref{eq:uA}--\eqref{eq:u11A}. In Section \ref{uni}, we prove the uniqueness result, and in Section \ref{sec:error} we discuss the convergence rates. Finally, in Section~\ref{num}, we present some numerical experiments which illustrate the theory.
		%----------------- SECTION
		\section{Definitions and notation}\label{def}
		Throughout this paper, we use the following notation:
		\begin{enumerate}
			% \item 
			\item For $\boldsymbol{z}:=(z^1,\ldots,z^N)\in \R^N,$ let $\norma{\boldsymbol{z}}:=\sum_{k=1}^N\abs{z^k}$ denote the usual $1$-norm.
			\item For $\boldsymbol{f}: = (f^1,\ldots, f^N)\in (L^1(\R))^N,$ $\norma{f}_{(L^1(\R))^N}:=\sum_{k=1}^N \norma{f^k}_{L^1(\R)}$  be the norm on the product space.
			\item For the matrix valued function $\boldsymbol{\mu} \in L^{\infty}(\R;\R^{N^2}),$
			$$
			\norma{\boldsymbol{\mu}}_{(L^{\infty}(\R))^{N^2}}:= \max\limits_{ 1 \leq i,j \leq N} \norma{\mu^{i,j}}_{L^{\infty}(\R)}.
			$$ 
			In addition, if $\boldsymbol{\mu}\in C^2(\R;\R^{N^2})$, then $\boldsymbol{\mu}'\in C^1(\R;\R^{N^2})$ and $\boldsymbol{\mu}'' \in C(\R;\R^{N^2})$ denote the component-wise derivative and second derivative, respectively. 
			\item For $u: \overline{Q}_T\rightarrow \R,$ $\boldsymbol{U}:\overline{Q}_T \rightarrow \R^N$ 
			and $\tau>0$, 
			\begin{align*}
				|u|_{\lip_tL^1_x}&:=\sup_{0\leq t_1<t_2\leq T}\frac{\norma{u(t_1)-u(t_2)}_{L^1(\R)}}{|t_1-t_2|},\\
				|\boldsymbol{U}|_{(\lip_tL^1_x)^N}&:=\max_{1 \leq k\leq N}|U^k|_{\lip_tL^1_x}, \\ 
				|\boldsymbol{U}|_{(L^\infty_t\bv_x)^N}&:=\max_{1 \leq k\leq N}\sup_{t\in[0,T]} TV(U^k(t,\dott)),\\
				\norma{\boldsymbol{U}}_{(L^{\infty}(\overline{Q}_T))^N}&:=\max_{1 \leq k\leq N} \norma{U^k}_{L^{\infty}(\overline{Q}_T)},\\
				\norma{\boldsymbol{U}}_{(L^1(\overline{Q}_T))^N}&:=\sum_{k=1}^{N} \norma{U^k}_{L^1(\overline{Q}_T) },\\
				\gamma(\boldsymbol{U},\tau)&:= \max\limits_{1 \leq k \leq N}\sup_{\substack{
						\abs{t_1-t_2} \leq \tau\\  0\leq t_1\leq t_2\leq T }} \norma{U^k(t_1)-U^k(t_2)}_{L^1(\R)}.
			\end{align*} 
		\end{enumerate}
		We will make the following assumptions throughout this work:
		\begin{enumerate}[(\textbf{H\arabic*})]
			\item \label{H1A}$f^k \in  \lip(\R)$  with $ f^k(0)=0$; 
			\item \label{H2A}$\nu^k \in (C^2 \cap   \bv  \cap \, W^{2,\infty}) (\R^N,\R), \boldsymbol{\mu}\in (C^2 \cap   \bv  \cap \, W^{2,\infty}) (\R,\R^{N^2})$, with $\nu^k(\boldsymbol{0})=0$;
			\item \label{H3A} $\sigma^k\in (\bv \cap L^1)(\R)$ with $\inf\limits_{x\in\R}\sigma^k(x)>0$.
		\end{enumerate}
		\begin{definition}
			A function 
			$\boldsymbol{U}\in (C([0,T];L^1(\R)) \cap L^{\infty}([0,T];\bv(\R)))^{N}$  is an entropy solution of IVP \eqref{eq:uA}--\eqref{eq:u11A}, 
			if for all $k \in \{1,\ldots,N\},$ and for all $\alpha\in \R$, 
			\begin{multline} \label{kruz2}
				%\begin{split}
				\int_{Q_T}\left|U^k(t,x)- \frac{\alpha}{\sigma^k(x)}\right|\phi_t(t,x)  \d x \d t  \\ 
				+ \int_{Q_T}\sgn (\overline{U}^k(t,x)-\alpha) \nu^k((\boldsymbol{\mu}*\boldsymbol{U})^k(t,x))
				(f^k(\overline{U}^k(t,x))-f^k(\alpha))\phi_x(t,x) \d{x} \d{t}\\ 
				-\int_{Q_T} f^k(\alpha) (\sgn (\overline{U}^k(t,x)-\alpha)) \partial_x\nu^k((\boldsymbol{\mu}*\boldsymbol{U})^k(t,x))\phi(t,x)\d{x} \d{t}\\ 
				+\int_{\R} \left|U_0^k(x)- \frac{\alpha}{\sigma^k(x)}\right|\phi(0,x)  \d x\geq 0, 
			\end{multline}
			for all non-negative $\phi\in C_c^{\infty}([0,T)\times \R)$.
		\end{definition}
		The following section  is dedicated to the finite volume schemes approximating the  above defined adapted entropy solutions for the IVP \eqref{eq:uA}--\eqref{eq:u11A}, and establishes  their existence via the convergence of the finite volume scheme.
		
		%--------------- section
		\section{Existence of the entropy solution}\label{sec:existence}
		For $\Delta x,\Delta t>0$ and $\lambda=\Delta t/\Delta x,$ consider equidistant spatial grid points $x_i:=i\Delta x$ for $i\in\Z$ and temporal grid points $t^n:=n\Delta t$ 
		for integers $n\in\{0,\ldots,N_T\}$, such that $T=N_T \D t.$  Let $\mathbbm{1}_i(x)$ denote the indicator function of $C_i:=[x_{i-1/2}, x_{i+1/2})$, where $x_{i+1/2}=\frac12(x_i+x_{i+1})$ and let
		$\chi^n(t)$ denote the indicator function of $C^{n}:=[t^n,t^{n+1})$, and let $C_i^n:=C^n\times C_i$. For each $k \in \{1, \ldots, N\},$ we approximate the initial data and $\sigma^k$ according to:
		\begin{align}\label{initial}
			U^{k}_{\Delta}(0,x)&:=\sum_{i\in\Z}\mathbbm{1}_i(x)U^{k,0}_i, \quad x\in \R, \text{ where }  U^{k,0}_i=\int_{C_i}U^{k}_0(x) \d x,\quad i\in\Z,\\
			\nonumber \sigma^{k}_{\Delta}(x)&:=\sum_{i\in\Z}\mathbbm{1}_i(x)\sigma^{k}_i, \quad x\in \R, \text{ where }  \sigma^{k}_i=\sigma^{k}(x_i),\quad i\in\Z.
		\end{align}
		Further, we define a piecewise constant approximate solution $\boldsymbol{U}_{\Delta} =
		\left(U_{\Delta}^{1}, \ldots, U_{\Delta}^{N}\right)$, i.e., $\boldsymbol{U}_{\Delta}\approx\boldsymbol{U}$, to the IVP~\eqref{eq:uA}--\eqref{eq:u11A} 
		by
		\begin{displaymath}
			U_{\Delta}^{k} (t,x) =  U^{k,n}_{i}
			\,\, \text{ for 
				$(t,x)  \in C_i^{n}$, where 
				$n  \in  \{0, \ldots, N_T\}, \, i  \in  \Z,
				\, k  \in  \{1, \ldots, N\}$.}
		\end{displaymath}
		
		The approximate solution $\boldsymbol{U}_{\Delta}$ is updated  through the following marching formula:
		\begin{align}
			U^{k,n+1}_i
			&:= U^{k,n}_i- \lambda \bigl[
			F^k (\nu^k({\boldsymbol{c}}^{k,n}_{i+1/2}),\overline{U}_i^{k,n},\overline{U}_{i+1}^{k,n})
			- 
			F^k (\nu^k({\boldsymbol{c}}^{k,n}_{i-1/2}),\overline{U}_{i-1}^{k,n},\overline{U}_{i}^{k,n})
			\bigr]\nonumber\\
			& := 
			U^{k,n}_i- \lambda \bigl[
			F^{k,n}_{i+1/2} (\overline{U}_i^{k,n},\overline{U}_{i+1}^{k,n})
			- 
			F^{k,n}_{i-1/2} (\overline{U}_{i-1}^{k,n}, \overline{U}_i^{k,n})
			\bigr]\label{scheme}\\
			&=H^k(\nu^k({\boldsymbol{c}}^{k,n}_{i-1/2}),\nu^k({\boldsymbol{c}}^{k,n}_{i+1/2}),\overline{U}_{i-1}^{k,n},\overline{U}_i^{k,n},\overline{U}_{i+1}^{k,n})\nonumber,
		\end{align}
		where  \\
		(1) for each $n  \in  \{0, \ldots, N_T\}, 
		i  \in  \Z$ and $
		k  \in  \{1, \ldots, N\}$, $\overline{U}_i^{k,n}:= \sigma^k_iU_i^{k,n}\approx \sigma^k(x_i)U^k(x_i,t^n)$ (pointwise value of the right continuous representative) and which is extended to the function defined in ${Q}_T$, via 
		\begin{equation*}
			\overline{U}^{k}_{\Delta}(t,x) =\sum_{n=0}^{N_T-1} \sum_{i\in\Z} \mathbbm{1}_i(x) \chi^n(t) \overline{U}_i^{k,n};
		\end{equation*}
		(2) the convolution term is updated through the following formula:
		\begin{align}\label{eq:conv1}
			{\boldsymbol{c}}^{k,n}_{ i+1/2}&=(c_{ i+1/2}^{1,k,n},\ldots, c_{ i+1/2}^{N,k,n}),
		\end{align}with, for each $j\in\{1,\ldots,N\}$, we have
		$c_{ i+1/2}^{j,k,n}$ being computed through the following quadrature formula,
		\begin{align*} 
			% \label{eq:conv1}
			c_{ i+1/2}^{j,k,n}&=\Delta x\sum\limits_{p\in \Z} \mu^{j,k}_{i+1/2-p}U^{j,n}_{p+1/2} \approx c^{j,k}(t^n,x_{i+1/2})=
			\int_{\R} \mu^{j,k}(x_{i+1/2}-y)U^{j}_{\D}(t^n,y)\d y, 
		\end{align*}
		with $U^{k,n}_{p+1/2}$ being any convex combination of
		$U^{k,n}_{p}$ and $U^{k,n}_{p+1}$, and $\mu^{j,k}_{i+1/2} = \mu^{j,k} (x_{i+1/2}).$  \\
		(3) $F^k$ is chosen so as to make $H^k$ increasing in its last three arguments. 
		
		The convolution terms satisfy the following estimates:
		\begin{lemma}
			\label{lem:AB} For for each $n  \in  \{0, \ldots, N_T\},~
			i  \in  \Z$ and $
			k  \in  \{1, \ldots, N\}$, the convolution term \eqref{eq:conv1} satisfies
			\begin{align*}
				\norma{{\boldsymbol{c}}^{k,n}_{i+1/2} - {\boldsymbol{c}}^{k,n}_{i-1/2}}
				& \leq 
				\mathcal{K}_1\Delta x, \,\\
				\norma{{\boldsymbol{c}}^{k,n}_{i+3/2} -2{\boldsymbol{c}}^{k,n}_{i+1/2}+ {\boldsymbol{c}}^{k,n}_{i-1/2}}
				& \leq 
				\mathcal{K}_2\, \Delta x^{2} \, \,,
			\end{align*}
			where \begin{align*}
				&\mathcal{K}_1=\norma{\boldsymbol{\mu'}}_{(L^{\infty}(\R))^{N^2}} \norma{\boldsymbol{U}_{\Delta}(t^n)}_{(L^1(\R))^N}, \\       &\mathcal{K}_2=2\norma{\boldsymbol{\mu''}}_{(L^{\infty}(\R))^{N^2}}
				\norma{\boldsymbol{U}_{\Delta}(t^n)}_{(L^1(\R))^N}.
			\end{align*}
		\end{lemma}
		\begin{proof}
			The proof follows by repeating the arguments of \cite[Prop.~2.8]{ACT2015} and \cite[Lemma~A.2]{ACG2015} in every component of the vector $\boldsymbol{c}_{i+1/2}^{k,n}.$
		\end{proof} 
		For each $
		k  \in  \{1, \ldots, N\}$, we modify the schemes proposed in \cite{ACG2015, AV2023} and  define the numerical fluxes that are the nonlocal extensions of the well-known  monotone fluxes for local conservation laws as described below: \\
		(1) \textbf{Lax--Friedrichs type flux:} 
		\begin{align*}
			F^{k}_{LF}(a,b,c)
			= 
			\frac{a}{2}\left( f^k(b)
			+
			f^k(c)\right)
			-
			\frac{\theta}{2\, \lambda}(c-b)\,,\quad \quad \theta \in \left(0,\frac{2}{3\norma{\sigma^k}_{L^\infty(\R)}} \right).
		\end{align*} 
		Further, $\Delta t$ is chosen in order to satisfy
		the following CFL condition:
		\begin{equation}\label{CFL_LF}
			\lambda \le \min_{k} 
			\left\{\frac{\min(1, 4-6\theta\norma{\sigma^k}_{L^\infty(\R)},6\theta\norma{\sigma^k}_{L^\infty(\R)})}{1+6\norma{\sigma^k}_{L^\infty(\R)}\abs{f^k}_{\lipR}\norma{\nu^k}_{L^\infty(\R)}}\right\}.
		\end{equation}
		(2) \textbf{Godunov type flux:}
		\begin{align*}
			F^{k}_{\rm Godunov}(a,b,c)= a F^k_{\rm Godunov}(b,c),
		\end{align*}
		where $ F^k_{\rm Godunov}$ is the Godunov flux for the corresponding local conservation law with discontinuous flux $f^k(x,u)=f^k(\sigma^k(x)u).$ Further, $\Delta t$ is chosen in order to satisfy
		the following CFL condition:
		\begin{equation}\label{CFL_God}
			\lambda \max_{k}\left\{{\norma{\sigma^k}_{L^\infty(\R)}}\abs{f^k}_{\lipR}\norma{\nu^k}_{L^\infty(\R)}\right\} \leq \frac{1}{6}.
		\end{equation}
		The scheme \eqref{scheme} can now be rewritten in an incremental form as (see \cite{AV2023, GTV2022a})
		\begin{align}
			\begin{split}
				\overline{U}_i^{k,n+1}
				&= \overline{U}_i^{k,n} + \overline{U}_i^{k,n+1}-\overline{U}_i^{k,n} \\
				&=  \overline{U}_i^{k,n} -\lambda  \sigma_i^k \left[F^{k,n}_{i+1/2}(\overline{U}_i^{k,n},\overline{U}_{i+1}^{k,n})-F^{k,n}_{i-1/2}(\overline{U}_{i-1}^{k,n}, \overline{U}_i^{k,n})\right]\\
				&= 
				\overline{U}_{{i}}^{k,n}
				-
				a_{i-1/2}^{k,n} \,  \Delta_-\overline{U}_{{i}}^{k,n}
				+
				b_{i+1/2}^{k,n} \,  \Delta_+\overline{U}_{{i}}^{k,n}\\
				&   \quad  - \lambda  \sigma_i^k\left(
				F^{k,n}_{i+1/2} (\overline{U}_{{i}}^{k,n},\overline{U}_{{i}}^{k,n})
				-
				F^{k,n}_{i-1/2} (\overline{U}_{{i}}^{k,n},\overline{U}_{{i}}^{k,n})
				\right), \label{incremental_form}
			\end{split}
		\end{align}   
		where  $\Delta_+z_{i}^{k,n}=z_{i+1}^{k,n}-z_{i}^{k,n}=\Delta_-z^{k,n}_{i+1},$ for any vector $z$, and
		\begin{align*}
			a_{i+1/2}^{k,n}
			& = 
			\lambda  \sigma_i^k \,
			\frac{
				{\color{black}F^{k,n}_{i+1/2} (\overline{U}^{k,n}_{i+1},\overline{U}^{k,n}_{i+1})}
				-
				F^{k,n}_{i+1/2}(\overline{U}^{k,n}_{i},\overline{U}^{k,n}_{i+1})}{\Delta_+\overline{U}_{{i}}^{k,n}} \,,\\
			b_{i+1/2}^{k,n}
			&= 
			\lambda  \sigma_i^k \,
			\frac{
				{\color{black}F^{k,n}_{i+1/2} (\overline{U}_{{i}}^{k,n},\overline{U}_{{i}}^{k,n})}
				-
				F^{k,n}_{i+1/2}(\overline{U}_{{i}}^{k,n},\overline{U}^{k,n}_{i+1})}{\Delta_+\overline{U}_{{i}}^{k,n}} \,.
		\end{align*} 
		The following theorem establishes the existence and the regularity of the entropy solutions of IVP~\eqref{eq:umulA}:
		\begin{theorem}\label{Existence}[Stability of the finite volume approximations]
			Assume that \ref{H1A}--\ref{H3A} hold. For $0<t\leq T$, fix a non-negative initial data $\boldsymbol{U}_0 \in ((\L1\cap
			\bv) (\R))^N$. Then, there exist constants $\mathcal{K}_i,~i=3,\ldots,6,$ independent of the mesh size $\D x$ such that the sequence of approximations $\boldsymbol{U}_{\Delta}$ defined by \eqref{scheme} under an appropriate CFL condition  (cf.~\eqref{CFL_LF} and \eqref{CFL_God} for the Lax--Friedrichs and Godunov type scheme, respectively), satisfy:  \\
			(1) Positivity:
			\begin{align}\label{apx:positivity}
				\boldsymbol{U}_{\Delta} (t,x) &\geq 0 \,.
			\end{align}
			(2) $L^{1}$ estimate:
			\begin{align}\label{apx:L1}
				\norma{ \boldsymbol{U}_{\Delta} (t)}_{(\L1(\R))^N}& = \norma{ \boldsymbol{U}_0}_{(\L1(\R))^N} \,.
			\end{align}
			(3) $L^{\infty}$ estimate:
			\begin{align}\label{apx:Linf}
				\norma{\boldsymbol{\overline{U}}_{\Delta}(t)}_{(\L\infty(\R))^N}
				& \leq
				\exp(\mathcal{K}_3\, t ) \,
				\norma{\boldsymbol{\overline{U}}_{\D}(0)}_{(\L\infty(\R))^N} \, .
			\end{align}
			Furthermore,
			\begin{align*}
				\norma{\boldsymbol{U}_{\Delta}(t)}_{(\L\infty(\R))^N}
				& \leq \frac{\norma{\boldsymbol{\sigma}}_{(L^\infty(\R))^N}}{\inf\limits_{x\in\R}\sigma^k(x)}
				\exp(\mathcal{K}_3\, t ) \,
				\norma{\boldsymbol{U}_0}_{(\L\infty(\R))^N} \, .
			\end{align*}
			(4) BV estimate: For each $n  \in  \{0, \ldots, N_T\}$ and $
			k  \in  \{1, \ldots, N\}$,
			\begin{align}\label{apx:bv}
				\sum_{i\in\mathbb{Z}}
				\modulo{\Delta_+\overline{U}_{{i}}^{k,n}} 
				&\leq 
				\exp(\mathcal{K}_4t)\left(\sum_{i\in\Z}\abs{\Delta_+\overline{U}_i^{k,0}}  + \mathcal{K}_5t\right).
			\end{align}
			(5) \label{lem:time} Time estimate:  For each $n  \in  \{0, \ldots, N_T-1\}$ and $
			k  \in  \{1, \ldots, N\}$,
			\begin{align}\label{apx:time}
				\Delta x\sum\limits_{i\in \mathbb{Z}} \abs{U_i^{k,n+1}-U_i^{k,n}} &\leq \mathcal{K}_6 \Delta t.
			\end{align}
			(6) \label{dis_ent}Discrete entropy inequality:
			For each $n  \in  \{0, \ldots, N_T-1\}, ~i\in\Z$, and $
			k  \in  \{1, \ldots, N\}$, and
			for all $\alpha \in \R, $ 
			\begin{align*} \modulo{U_i^{k,n+1}-  \frac{\alpha}{\sigma^k_i}}&-\modulo{U_i^{k,n}-   \frac{\alpha}{\sigma^k_i}}+\lambda\big(G^{k,n}_{i+1/2}(U_i^{k,n} ,U_{i+1}^{k,n},\alpha)-G^{k,n}_{i-1/2}(U_{i-1}^{k,n} ,U_i^{k,n},\alpha)\big)\\
				& \quad+\lambda\sgn\left(U_i^{k,n+1}-   \frac{\alpha}{\sigma^k_i}\right) f^k(\alpha)(\nu^k(\boldsymbol{c}_{i+1/2}^{k,n})-\nu^k(\boldsymbol{c}_{i-1/2}^{k,n}))\le 0,
			\end{align*} 
			where 
			\begin{align*}
				G^{k,n}_{i+1/2} (p,q,\alpha)
				&:= F^{k,n}_{i+1/2}\left( \sigma^k_i\left(p\vee   \frac{\alpha}{\sigma^k_i}\right), \sigma^k_{i+1}\left(q\vee  \frac{\alpha}{\sigma^k_{i+1}}\right)\right)
				\\&\quad-
				F^{k,n}_{i+1/2}\left( \sigma^k_i\left(p\wedge  \frac{\alpha}{\sigma^k_i}\right), \sigma^k_{i+1}\left(q\wedge  \frac{\alpha}{\sigma^k_{i+1}}\right)\right ) \,,
			\end{align*}
			with $u\vee v=\max(u,v), u\wedge v=\min(u,v).$ \\
			(7) \label{sub_conc} Up to a subsequence, the numerical approximations $\boldsymbol{U}_{\D}$  converge to an entropy solution $ \boldsymbol{U} \in (C([0,T];L^1(\R)) \cap L^{\infty}([0,T];\bv(\R)))^{N}.$
		\end{theorem}
		\begin{proof} The proof of the above theorem follows by invoking the incremental form of the scheme and can be done in the spirit of \cite{ACG2015,ACT2015,AV2023} with appropriate modifications.  We skip the detailed proof and present only the key estimates below. Let $n  \in  \{0, \ldots, N_T\},~i\in \Z$ and $
			k  \in  \{1, \ldots, N\}$.  \\
			(1)  Using the CFL condition, cf.~\eqref{CFL_LF} and \eqref{CFL_God}, it is straightforward to see that the incremental coefficients satisfy,
			$0\le a^{k,n}_{i+1/2}, b^{k,n}_{i+1/2}
			\leq
			\frac{1}{3}.$
			Moreover, using $U^{k,n} \ge 0$, $\overline{U}^{k,n}\ge 0$ and Lemma~\ref{lem:AB}, the CFL condition 
			(cf.~\eqref{CFL_LF} and \eqref{CFL_God}) can again be invoked to show:
			\begin{align*}
				\sigma_{i}^{k} \modulo{ F^{k,n}_{i+1/2} (\overline{U}_{{i}}^{k,n},\overline{U}_{{i}}^{k,n})
					-
					F^{k,n}_{i-1/2} (\overline{U}_{{i}}^{k,n},\overline{U}_{{i}}^{k,n})}
				&=
				\sigma_{i}^{k} \overline{U}_i^{k,n}\left|
				\nu^k(c_{i+1/2}^{k,n})
				-
				\nu^k(c_{i-1/2}^{k,n})
				\right| \\
				&\leq 
				2  \sigma_{i}^{k}\overline{U}_i^{k,n}\norma{\nu^k}_{L^\infty(\R)}\\
				&\leq
				\frac{1}{3\lambda}\overline{U}_i^{k,n}.  
			\end{align*}
			Finally,
			\begin{align*}
				\overline{U}_{i}^{k,n+1}
				&= 
				(1-a^{k,n}_{i-1/2} -b^{k,n}_{i+1/2}) \overline{U}_{{i}}^{k,n}
				+
				a^{k,n}_{i-1/2} \overline{U}_{i-1}^{k,n}
				+
				b^{k,n}_{i+1/2} \overline{U}_{i+1}^{k,n}
				\\
				\nonumber
				& \quad 
				-
				\lambda     \sigma_{i}^{k}\left(
				F^{k,n}_{i+1/2} (\overline{U}_{{i}}^{k,n},\overline{U}_{{i}}^{k,n})
				-
				F^{k,n}_{i-1/2} (\overline{U}_{{i}}^{k,n},\overline{U}_{{i}}^{k,n})
				\right)
				\\
				&\geq 
				\left(\frac{2}{3}-a^{k,n}_{i-1/2} - b^{k,n}_{i+1/2}   \right) \overline{U}_{{i}}^{k,n}
				+
				a^{k,n}_{i-1/2} \overline{U}_{i-1}^{k,n}
				+
				b^{k,n}_{i+1/2} \overline{U}_{i+1}^{k,n}
				\; \geq \;
				0 \,.
			\end{align*}
			Since $\sigma^k> 0,$ the result follows.  \\
			(2) Since the scheme is conservative, the result follows from the positivity of the scheme, cf.~\eqref{apx:positivity}. \\
			(3)
			Using Lemma~\ref{lem:AB}, we obtain:
			\begin{align}
				F^{k,n}_{i+1/2} (\overline{U}_i^{k,n},\overline{U}_i^{k,n})
				-
				F^{k,n}_{i-1/2} (\overline{U}_i^{k,n},\overline{U}_i^{k,n})&=
				f^k( \overline{U}_i^{k,n}) \left(\nu^k(c_{i+1/2}^{k,n})
				-
				\nu^k(c_{i-1/2}^{k,n})\right)\nonumber
				\\ \label{eq:bound1}
				&\leq 
				\abs{f^k}_{\lipR} \overline{U}_{{i}}^{k,n}
				\abs{\nu^k}_{\lipR}      \mathcal{K}_1 \Delta x.
			\end{align}
			Now, using the incremental form of the scheme, cf.~\eqref{incremental_form},
			and the inequality \eqref{eq:bound1}, we get
			\begin{align*}
				\modulo{ \overline{U}_i^{k,n+1}}
				& \leq 
				(1-a^{k,n}_{i-1/2}-b^{k,n}_{i+1/2} ) \overline{U}_i^{k,n}
				+
				a^{k,n}_{i-1/2} \overline{U}_{i-1}^{k,n}
				\\
				& \qquad +  b^{k,n}_{i+1/2} \overline{U}_{i+1}^{k,n}
				+ \Delta t\norma{ \sigma^k}_{\L\infty(\R)}
				\abs{f^k}_{\lipR} 
				\abs{\nu^k}_{\lipR}      \mathcal{K}_1 \overline{U}_{{i}}^{k,n}
				\\
				& \leq 
				(1-a^{k,n}_{i-1/2}-b^{k,n}_{i+1/2}) \norma{ \overline{U}^k_{\Delta}(t^n)}_{\L\infty(\R)}\\
				&\quad
				+
				a^{k,n}_{i-1/2} \norma{ \overline{U}^k_{\Delta}(t^n)}_{\L\infty(\R)}
				+
				b^{k,n}_{i+1/2} \norma{\overline{U}^k_{\Delta}(t^n)}_{\L\infty(\R)}\\
				&\quad
				+ \Delta t\norma{ \sigma^k}_{\L\infty(\R)}
				\abs{f^k}_{\lipR} 
				\abs{\nu^k}_{\lipR}      \mathcal{K}_1 \norma{\overline{U}^k_{\Delta}(t^n)}_{\L\infty(\R)}
				\\
				&\leq 
				\norma{\overline{U}^k_{\Delta}(t^n)}_{\L\infty(\R)}
				\left(
				1
				+
				\Delta t\mathcal{K}_1\norma{ \sigma^k}_{\L\infty(\R)}
				\abs{f^k}_{\lipR} 
				\abs{\nu^k}_{\lipR}\right)\,,
			\end{align*}
			which implies that
			\begin{align*}
				\norma{\overline{U}^k_{\Delta}(t^{n})}_{\L\infty(\R)}
				&\leq
				\norma{\overline{U}^k_{\Delta}(0)}_{\L\infty(\R)}
				\left(1 + \Delta t\mathcal{K}_1\norma{ \sigma^k}_{\L\infty(\R)}
				\abs{f^k}_{\lipR} 
				\abs{\nu^k}_{\lipR}\right) ^{n}\\
				&\le \exp(\mathcal{K}_3 \, t^n ) \,
				\norma{\overline{U}^k_{\Delta}(0)}_{\L\infty(\R)},
			\end{align*}
			so that~\eqref{apx:Linf} holds. 
			Furthermore, since $\boldsymbol{\overline{U}}_{\Delta}$ is bounded and $ \sigma^k>0,$ the approximate function $\boldsymbol{U}_{\Delta}$ is also bounded. \\
			(4) Because of the incremental form of the scheme, we get, by taking the forward difference in the spatial direction, that 
			\begin{align*}
				\Delta_+\overline{U}_{i}^{k,n+1}
				&=
				\Delta_+\overline{U}_{{i}}^{k,n}\left(1-
				a_{i+1/2}^{k,n}-b_{i+1/2}^{k,n} \right)       +a_{i-1/2}^{k,n} \, \Delta_-\overline{U}_{{i}}^{k,n}
				+
				b_{i+3/2}^{k,n} \, \Delta_+\overline{U}_{i+1}^{k,n} \\ 
				& \qquad  -
				\lambda    \sigma_{i+1}^{k}\left(
				F^{k,n}_{i+3/2} (\overline{U}^{k,n}_{i+1},\overline{U}^{k,n}_{i+1})
				-
				F^{k,n}_{i+1/2} (\overline{U}^{k,n}_{i+1},\overline{U}^{k,n}_{i+1})
				\right) \\\nonumber 
				& \qquad +\lambda    \sigma_{i}^{k}\left(
				F^{k,n}_{i+1/2} (\overline{U}^{k,n}_{i},\overline{U}^{k,n}_{i})
				-
				F^{k,n}_{i-1/2} (\overline{U}^{k,n}_{i},\overline{U}^{k,n}_{i})
				\right)\\
				&\nonumber:= A_1+A_2+A_3.
			\end{align*}
			Since the incremental coefficients satisfy $0\le    a_{i+1/2}^{k,n},   b_{i+1/2}^{k,n}\le \frac13,$ we can repeat  Harten's argument (see \cite[Lemma~3.12]{HR2015}) to get 
			\begin{equation*}
				\sum_{i\in\Z}\modulo{A_1}\le\sum_{i\in\Z}\modulo{\Delta_+\overline{U}_{{i}}^{k,n}} .
			\end{equation*}
			Adding and subtracting $ \lambda  \sigma_{i}^{k}f^k(\overline{U}^{k,n}_{i})\left(
			\nu^k({\boldsymbol{c}}^{k,n}_{i+3/2})-
			\nu^k({\boldsymbol{c}}^{k,n}_{i+1/2})
			\right)$, we obtain
			\begin{align*}
				A_2+A_3
				&=-\lambda(
				\sigma_{i+1}^{k}f^k(\overline{U}^{k,n}_{i+1})-
				\sigma_{i}^{k}f^k(\overline{U}^{k,n}_{i}))\left(
				\nu^k({\boldsymbol{c}}^{k,n}_{i+3/2}) -
				\nu^k({\boldsymbol{c}}^{k,n}_{i+1/2})
				\right) \\
				&\quad -\lambda   \sigma_{i}^{k}f^k(\overline{U}^{k,n}_{i})\left(
				\nu^k({\boldsymbol{c}}^{k,n}_{i+3/2})-
				2\nu^k({\boldsymbol{c}}^{k,n}_{i+1/2})
				+
				\nu^k({\boldsymbol{c}}^{k,n}_{i-1/2})
				\right).
			\end{align*}
			Further,
			\begin{align*}\nonumber
				|A_2+A_3|
				&  \le  \nonumber\lambda \modulo{  \sigma_{i+1}^{k}f^k(\overline{U}^{k,n}_{i+1})-  \sigma_{i}^{k}f^k(\overline{U}^{k,n}_{i})}\abs{\nu^k}_{\lipR}\norma{{\boldsymbol{c}}^{k,n}_{i+3/2}-
					{\boldsymbol{c}}^{k,n}_{i+1/2}} \\
				& \quad +  \nonumber
				\lambda   \left|\sigma_{i}^{k}\right|\abs{f^k}_{\lipR}\left|\overline{U}^{k,n}_{i}\right|\left|
				\nu^k({\boldsymbol{c}}^{k,n}_{i+3/2})-
				2\nu^k({\boldsymbol{c}}^{k,n}_{i+1/2})
				+
				\nu^k({\boldsymbol{c}}^{k,n}_{i-1/2})
				\right|.
			\end{align*}
			Consider next
			\begin{multline*}
				\modulo{  \sigma_{i+1}^{k}f^k(\overline{U}^{k,n}_{i+1})-  \sigma_{i}^{k}f^k(\overline{U}^{k,n}_{i})} \\
				\le\abs{f^k}_{\lipR}\modulo{\Delta_+\overline{U}_{{i}}^{k,n}}\norma{ \sigma^k_{\Delta}}_{L^{\infty}(\R)}+\abs{f^k}_{\lipR}\norma{\overline{U}^k_{\Delta}}_{L^{\infty}(\R)}\modulo{\Delta_+  \sigma_{i}^{k}}.
			\end{multline*}
			Now we estimate the term
			\begin{multline*}
				\nu^k({\boldsymbol{c}}^{k,n}_{i+3/2})-
				2\nu^k({\boldsymbol{c}}^{k,n}_{i+1/2})
				+
				\nu^k({\boldsymbol{c}}^{k,n}_{i-1/2})\\
				= \nabla\nu^k({\boldsymbol{\xi}}^{k,n}_{i+1})\dott({\boldsymbol{c}}^{k,n}_{i+3/2}-{\boldsymbol{c}}^{k,n}_{i+1/2})-
				\nabla\nu^k({\boldsymbol{\xi}}^{k,n}_{i})\dott({\boldsymbol{c}}^{k,n}_{i+1/2}-{\boldsymbol{c}}^{k,n}_{i-1/2})
			\end{multline*}
			where  $\nabla$ denotes the gradient and ${\boldsymbol{\xi}}^{k,n}_{i+1}\in 
			I({\boldsymbol{c}}^{k,n}_{i+1/2},{\boldsymbol{c}}^{k,n}_{i+3/2})$ (here $I(\boldsymbol{a},\boldsymbol{b})$ denotes a rectangular box with sides parallel to the coordinate planes and $\boldsymbol{a}$ and $\boldsymbol{b}$ at opposite corners) and  ${\boldsymbol{\xi}}^{k,n}_{i}\in I({\boldsymbol{c}}^{k,n}_{i-1/2},{\boldsymbol{c}}^{k,n}_{i+1/2}).$
			Now, adding and subtracting $\nabla\nu^k({\boldsymbol{\xi}}^{k,n}_{i+1})\dott({\boldsymbol{c}}^{k,n}_{i+1/2}-{\boldsymbol{c}}^{k,n}_{i-1/2}),$ we have
			\begin{align*}
				\nu^k({\boldsymbol{c}}^{k,n}_{i+3/2})-
				& 2\nu^k({\boldsymbol{c}}^{k,n}_{i+1/2})
				+
				\nu^k({\boldsymbol{c}}^{k,n}_{i-1/2}) \\
				&=
				\nabla\nu^k({\boldsymbol{\xi}}^{k,n}_{i+1})\dott({\boldsymbol{c}}^{k,n}_{i+3/2}-2{\boldsymbol{c}}^{k,n}_{i+1/2}+{\boldsymbol{c}}^{k,n}_{i-1/2})\\
				&\quad-
				(\nabla\nu^k({\boldsymbol{\xi}}^{k,n}_{i+1})-\nabla\nu^k({\boldsymbol{\xi}}^{k,n}_{i}))\dott({\boldsymbol{c}}^{k,n}_{i+1/2}-{\boldsymbol{c}}^{k,n}_{i-1/2}).
			\end{align*}
			This implies, using Lemma \ref{lem:AB},
			\begin{align}\label{V1}
				\begin{split}
					& \left|(\nabla\nu^k({\boldsymbol{\xi}}^{k,n}_{i+1})-\nabla\nu^k({\boldsymbol{\xi}}^{k,n}_{i}))\dott({\boldsymbol{c}}^{k,n}_{i+1/2}-{\boldsymbol{c}}^{k,n}_{i-1/2})\right|\\&
					\qquad\le
					\abs{\nabla \nu^k}_{\lipR}
					\norma{\boldsymbol{\xi}^{k,n}_{i+1}-{\boldsymbol{\xi}^{k,n}_{i}}}\norma{{\boldsymbol{c}}^{k,n}_{i+1/2}-2{\boldsymbol{c}}^{k,n}_{i-1/2}}\\
					&\qquad\le \abs{\nabla \nu^k}_{\lipR} \norma{\boldsymbol{\xi}^{k,n}_{i+1}-{\boldsymbol{\xi}^{k,n}_{i}}}\mathcal{K}_1\Delta x. 
				\end{split}
			\end{align}
			Now using Lemma \ref{lem:AB},    
			\begin{align}\label{V2}
				\begin{split}
					\norma{\boldsymbol{\xi}^{k,n}_{i+1}-{\boldsymbol{\xi}^{k,n}_{i}}}&=
					\norma{({\boldsymbol{\xi}}^{k,n}_{i+1}-{\boldsymbol{c}}^{k,n}_{i+1/2})+({\boldsymbol{c}}^{k,n}_{i+1/2}-{\boldsymbol{\xi}}^{k,n}_{i})}\\
					&\le  \norma{{\boldsymbol{\xi}}^{k,n}_{i+1}-{\boldsymbol{c}}^{k,n}_{i+1/2}}+\norma{{\boldsymbol{c}}^{k,n}_{i+1/2}-{\boldsymbol{\xi}}^{k,n}_{i}}\\
					&\le \norma{{\boldsymbol{c}}^{k,n}_{i+3/2}-{\boldsymbol{c}}^{k,n}_{i+1/2}}+\norma{{\boldsymbol{c}}^{k,n}_{i+1/2}-{\boldsymbol{c}}^{k,n}_{i-1/2}}\\
					&\le  2\mathcal{K}_1\Delta x.
			\end{split}\end{align}
			Finally, using  Lemma \ref{lem:AB}, as well as \eqref{V1} and \eqref{V2}, we have
			\begin{align}\label{V}
				\begin{split}
					& \norma{\nu^k({\boldsymbol{c}}^{k,n}_{i+3/2})-
						2\nu^k({\boldsymbol{c}}^{k,n}_{i+1/2})
						+
						\nu^k({\boldsymbol{c}}^{k,n}_{i-1/2})}\\
					&\qquad\le \norma{\nabla\nu^k({\boldsymbol{\xi}}^{k,n}_{i+1})}\norma{({\boldsymbol{c}}^{k,n}_{i+3/2}-2{\boldsymbol{c}}^{k,n}_{i+1/2}+{\boldsymbol{c}}^{k,n}_{i-1/2})}\\
					&\qquad\quad+
					\norma{\nabla\nu^k({\boldsymbol{\xi}}^{k,n}_{i+1})-\nabla\nu^k({\boldsymbol{\xi}}^{k,n}_{i})}{\norma{{\boldsymbol{c}}^{k,n}_{i+1/2}-{\boldsymbol{c}}^{k,n}_{i-1/2})}}\\
					&\qquad\le  \abs{\nu^k}_{\lipR}\mathcal{K}_2\Delta x^2+ 2 \abs{D \nu^k}_{\lipR}(\mathcal{K}_1)^2(\Delta x)^2.
				\end{split} 
			\end{align}
			Consequently,
			\begin{align*}
				\sum_{i\in\Z}\left|A_2+A_3\right|
				&\le\lambda\sum_{i\in\Z} \modulo{  \sigma_{i+1}^{k}f^k(\overline{U}^{k,n}_{i+1})-  \sigma_{i}^{k}f^k(\overline{U}^{k,n}_{i})}\abs{\nu^k}_{\lipR}\mathcal{K}_1\Delta x \\&\quad+
				\lambda \Delta x^2\abs{f^k}_{\lipR}\norma{ \sigma^k_{\Delta}}_{L^{\infty}(\R)}( \abs{\nu^k}_{\lipR}\mathcal{K}_2+ 2 \abs{D \nu^k}_{\lipR}(\mathcal{K}_1)^2)\sum_{i\in\Z}\overline{U}_i^{k,n}\\
				& \le  \Delta t\abs{f^k}_{\lipR}\norma{ \sigma^k_{\Delta}}_{L^{\infty}(\R)}\abs{\nu^k}_{\lipR}\mathcal{K}_1\sum_{i}\modulo{\Delta_+\overline{U}_{{i}}^{k,n}}\\
				&\quad +\Delta t\abs{f^k}_{\lipR}\norma{\overline{U}^k_{\Delta}}_{L^{\infty}(\R)}\abs{\nu^k}_{\lipR}\mathcal{K}_1\modulo{ \sigma^k_{\Delta}}_{BV(\R)} \\
				&\quad + \Delta t\abs{f^k}_{\lipR}\norma{ \sigma^k_{\Delta}}_{L^{\infty}(\R)}( \abs{\nu^k}_{\lipR}\mathcal{K}_2+ 2 \abs{D \nu^k}_{\lipR}(\mathcal{K}_1)^2)\norma{\overline{U}^k_{\Delta}}_{L^1(\R)}\\
				&=\Delta t     \sum_{i\in\Z}|\Delta_+\overline{U}_i^{k,n}| \mathcal{K}_4
				+ \Delta t\mathcal{K}_5,
			\end{align*}
			where,
			\begin{align*}
				\mathcal{K}_4&=\abs{f^k}_{\lipR}\abs{\nu^k}_{\lipR}\norma{ \sigma^k_{\Delta}}_{\L\infty({\mathbb{R}})}\mathcal{K}_1,\\ \mathcal{K}_5 &=\mathcal{K}_1\abs{ \sigma^k_{\Delta}}_{BV(\R)}\abs{f^k}_{\lipR}\abs{\nu^k}_{\lipR}\norma{\overline{U}^k_{\Delta}}_{\L\infty({\mathbb{R}})}\\&\quad +\abs{f^k}_{\lipR}\norma{ \sigma^k_{\Delta}}_{L^{\infty}(\R)}( \abs{\nu^k}_{\lipR}\mathcal{K}_2+ 2 \abs{D \nu^k}_{\lipR}(\mathcal{K}_1)^2)\norma{\overline{U}^k_{\Delta}}_{L^1(\R)}.\end{align*}
			Finally,
			using the above estimates, we get:
			\begin{align*}
				\sum_{i\in\Z}\abs{\Delta_+ \overline{U}_i^{k,n+1}}
				\le  \sum_{i\in\Z}\abs{\Delta_+\overline{U}_i^{k,n}}   (1+\mathcal{K}_4\Delta t  ) + \Delta t\mathcal{K}_5, 
			\end{align*}
			completing the proof. \\
			(5) Owing to the total variation bound of $\boldsymbol{\overline{U}}_{\Delta}$, and following \cite{ACG2015,AV2023},
			the proof follows by repeating the arguments in every component of the vector.  \\
			(6) For $k
			\in \{1, \ldots, N\} $ fixed, the discrete entropy inequality follows by 
			repeating the arguments in \cite{ACG2015,AV2023} in every component of the vector $\boldsymbol{U}_{\Delta}$, using appropriate bounds on the matrix $\boldsymbol{c}.$  \\
			(7) Follows by repeating the Lax--Wendroff type argument (see \cite[Thm.~5.1]{BBKT2011}) equation by equation.
			%\end{enumerate}
		\end{proof}
		The following corollary is a direct consequence of the above theorem.
		\begin{corollary}[Regularity]
			For $0<t\leq T,$ the limit $\boldsymbol{U}$ of the numerical approximation $\boldsymbol{U}_{\Delta}$ which is an entropy solution of the IVP~\eqref{eq:uA}--\eqref{eq:u11A} for the system 
			of nonlocal conservation laws, satisfies the following:
			\begin{align*}
				\norma{\boldsymbol{U}(t,\dott)}_{(L^{\infty}(\R))^N}&\leq \frac{\norma{\boldsymbol{\sigma}}_{(L^\infty(\R))^N}}{\min\limits_{k\in\{1\ldots N\}}\inf\limits_{x\in\R}\sigma^k(x)}\exp(\mathcal{K}_3 T)\norma{\boldsymbol{U}_0(\dott)}_{(L^{\infty}(\R))^N},\\
				\norma{\boldsymbol{U}(t,\dott)}_{(L^{1}(\R))^N}&=\norma{\boldsymbol{U}_0}_{(L^{1}(\R))^N},\\
				\abs{\boldsymbol{\sigma}(\dott)\circ\boldsymbol{U}(t,\dott)}_{(\bv(\R))^N} &\leq \exp(\mathcal{K}_4t)(\TV( s(\dott)\boldsymbol{U}_0) +\mathcal{K}_5 t), \\
				\norma{\boldsymbol{U}(t_2,\dott)-\boldsymbol{U}(t_1,\dott)}_{(L^1(\R))^N} &\leq N\mathcal{K}_6\abs{t_2-t_1}.
			\end{align*}
		\end{corollary}
		In the next section, we prove that the entropy solutions are in fact unique.
		
		%---------------- SECTION
		\section{Kuznetsov-type estimate and the uniqueness}\label{uni}
		% ** Define all the norms**
		For $\epsilon,\epsilon_0>0,$ define $\Phi: \overline{Q}_T^2 \rightarrow \R$ by  
		\begin{align*}
			\Phi(t,x,s,y):=\Phi^{\epsilon,\epsilon_0}(t,x,s,y)=\omega_{\epsilon}(x-y)\omega_{{\epsilon}_0}(t-s),
		\end{align*}
		where $\omega_a(x)=\frac{1}{a}\omega\left(\frac{x}{a}\right),$ $a>0$ and $\omega$ is a standard symmetric mollifier with $\operatorname{supp} (\omega) \in [-1,1].$ Furthermore, we assume that $\int_{\R} \omega_a(x) \d x =1$ and $\int_{\R} \abs{\omega'_a(x)} \d x =\frac{1}{a}.$ It is straight forward to see that $\Phi$ is symmetric,
		$\Phi_x=-\Phi_y$ and $\Phi_t=-\Phi_s$. Further, \begin{align}
			\int_{Q_T} \left(\int_{C^{n}} \Phi(s,y,t,x_{i+1/2})\d t-\lambda\int_{C_i}\Phi(s,y,t^{n+1},x)\d x \right) \d s \d y&=\mathcal{O} \left(\frac{\D x^2}{\epsilon} +\frac{\D t^2}{\epsilon_0}\right),\label{est:3}\\ 
			\int_{Q_T}\left(\int_{C_{i+1}}\Phi(s,y,t^{n+1},x)\d x-\int_{C_i}\Phi(s,y,t^{n+1},x)\d x \right) \d s \d y&= \mathcal{O}\left( \frac{\D x^2}{\epsilon}\right)\label{est:4}.
		\end{align}
		Details of these estimates can be found in \cite[Sec.~3.3]{HR2015}.
		We now list some of the notation to be used in the sequel. For $\boldsymbol{U},\boldsymbol{V} \in ((L^1 \cap L^{\infty})(\R))^N$, all non-negative $\phi\in C_c^{\infty}({\overline{Q}}_T),$ for each $(t,x),(s,y)\in Q_T$ and for each 
		$k\in \{1,\dots,N\},$ we define the following:
		\begin{align*}
			\mathcal{U}^k(t,x)&:=\nu^k((\boldsymbol{\mu}*\boldsymbol{U})^k(t,x)),\quad \quad\quad \quad\quad(t,x)\in Q_T,\\
			\mathcal{V}^k(s,y)&:=\nu^k((\boldsymbol{\mu}*\boldsymbol{V})^k(s,y)),\quad \quad \quad\quad\quad(s,y)\in Q_T,\\
			G^k(a,b)&:=\sgn (a-b) (f^k(a)-f^k(b)),\quad \quad a,b\in\R,\\
			\Lambda^k_T\left(U^k,\phi, \frac{\alpha}{\sigma^k}\right)&:= \int_{Q_T}\left(\Bigg|U^k(t,x)- \frac{\alpha}{\sigma^k(x)}\Bigg|\phi_{t}+G^k\left(U^k(t,x), \frac{\alpha}{\sigma^k(x)}\right)\mathcal{U}^k(t,x)\phi_{x}\right)\d t \d x\\
			&\quad -   \int_{Q_T}\sgn\left(U^k(t,x)- \frac{\alpha}{\sigma^k(x)}\right) f^k(\alpha)\mathcal{U}^k_x(t,x)\phi \d t \d x\\
			& \quad -\int_{\R}\left|U^k(T,x)- \frac{\alpha}{\sigma^k(x)}\right|\phi(T,x)\d x+\int_{\R}\left|U_0^k(x)- \frac{\alpha}{\sigma^k(x)}\right|\phi(0,x)\d x,\\[2mm] \Lambda^k_{\epsilon,\epsilon_0}(U^k,V^k)&:=\int_{Q_T}\Lambda^k_T\left(U^k(\dott,\dott),\Phi(\dott,\dott,s,y), \frac{\overline{V}^k(s,y)}{ \sigma^{k}(\dott)}\right)\d s \d y, 
			&\intertext{where}
			\overline{V}^k(s,y)&= \sigma^k(y)V^k(s,y),\\      
			K&:=\big\{\boldsymbol{U}:\overline{Q}_T \rightarrow \R^N:\quad\norma{\boldsymbol{U}}_{(L^{\infty}(\overline{Q}_T))^N}+|\boldsymbol{U}|_{(L^{\infty}([0,T];\bv(\R)))^{N}}<\infty \\& \qquad\text{ and }\norma{U^k(t)}_{L^1(\R) }=\norma{U^k(0)}_{L^1(\R)} \text{ for }t\ge 0, 1\leq k \leq N \big\}.
		\end{align*}
		We now state and prove the Kuznetsov-type lemma for the coupled system of nonlocal conservation laws ~\eqref{eq:umulA}.
		\begin{lemma}\label{lemma:kuz}[A Kuznetsov-type lemma for nonlocal systems of conservation laws]\\
			Let $\boldsymbol{U}$ be an entropy solution of the IVP of the system ~\eqref{eq:uA}--\eqref{eq:u11A}  and let $\boldsymbol{V} \in K.$ Then, 
			\begin{align}
				\label{est:kuz}
				\begin{split}
					& \norma{\boldsymbol{U}(T,\dott)-\boldsymbol{V}(T,\dott)}_{(L^1(\R))^N} \\&\qquad \le{\mathcal{C}}\sum_{k=1}^N\big( \gamma(V^k,\epsilon_0)-\Lambda^k_{\epsilon,\epsilon_0}(V^k,U^k) \big)  \\&\qquad\quad
					+\mathcal{C}\left( \norma{\boldsymbol{U}_0-\boldsymbol{V}_0}_{(L^1(\R))^N}+N\big(\epsilon+\epsilon_0 \big)+N \frac{\epsilon}{\epsilon_0}\left|\frac{1}{ s}\right|_{(\bv(\R))^N}\right),
				\end{split}
			\end{align}
			where $\mathcal{C}$ depends only on $\boldsymbol{f},\boldsymbol{\mu},\boldsymbol{\nu},\boldsymbol{\sigma}, \boldsymbol{U}, \boldsymbol{V}$, and $T.$
		\end{lemma}
		\begin{proof} For any $k\in \{1,\dots,N\}$, 
			add the entropy functionals $\Lambda^k_{\epsilon,\epsilon_0}(U^k,V^k)$\\ and $\Lambda^k_{\epsilon,\epsilon_0}(V^k,U^k),$ and invoke the symmetry of $\Phi$ to get:
			\begin{align*}
				\Lambda^k_{\epsilon,\epsilon_0}(U^k,V^k)+\Lambda^k_{\epsilon,\epsilon_0}(V^k,U^k)&=I^k_{\Phi_t}+
				I^k_{\Phi_x}+I^k_{\Phi}+I^k_0-I^k_T,\end{align*}  with \\
			\begin{align*}
				I_{\Phi_t}^k&= \int_{Q_T^2}\left|U^k(t,x)-\frac{\overline{V}^k(s,y)}{ \sigma^k(x)}\right| \Phi_t(t,x,s,y) \d{x} \d{t} \d{y} \d{s}\\
				&\quad-\int_{Q_T^2}\left|V^k(s,y)-\frac{\overline{U}^k(t,x)}{ \sigma^k(y)}\right|\Phi_t(t,x,s,y) \d{x} \d{t} \d{y} \d{s},\\[2mm]
				I^k_{\Phi_x}&=\int_{Q_T^2}G^k(\overline{U}^k(t,x),\overline{V}^k(s,y))(\mathcal{U}^k(t,x)-\mathcal{V}^k(s,y)) \Phi_{x}(t,x,s,y)\d t \d x \d s \d y,\\[2mm]
				I^k_{\Phi}&=-\int_{Q^2_T} \sgn (\overline{U}^k(t,x)-\overline{V}^k(s,y))f^k(\overline{V}^k(s,y))  \mathcal{U}^k_x(t,x)\Phi(t,x,s,y) \d x \d t  \d y \d s\\
				&\quad +\int_{Q^2_T}\sgn (\overline{U}^k(t,x)-\overline{V}^k(s,y))f^k(\overline{U}^k(t,x))  \mathcal{V}^k_y(s,y)\Phi(t,x,s,y) \d x \d t  \d y \d s,\\[2mm]
				I^k_T&=\int_{Q_T}\int_{\R}\left|U^k(T,x)-\frac{\overline{V}^k(s,y)}{ \sigma^k(x)}\right|\Phi(t,x,T,y)\d y  \d x \d t\\
				&\quad+\int_{Q_T}\int_{\R}\left|V^k(T,y)-\frac{\overline{U}^k(t,x)}{ \sigma^k(y)}\right|\Phi(t,x,T,y)\d y  \d x \d t,\\[2mm]
				I^k_0&=\int_{Q_T}\int_{\R}\left(\left|U^k_0(x)-\frac{\overline{V}^k(s,y)}{ \sigma^k(x)}\right|+\left|V^k_0(y)-\frac{\overline{U}^k(t,x)}{ \sigma^k(y)}\right|\right)\Phi(t,x,0,y)\d x \d y \d t.
			\end{align*}
			Since $\boldsymbol{U}$ is the entropy solution of \eqref{eq:uA}--\eqref{eq:u11A} , we have that $\Lambda^k_{\epsilon,\epsilon_0}(U^k,V^k)\ge 0, $ and hence
			\begin{align}
				\label{kuz}
				I^k_T &\le-  \Lambda^k_{\epsilon,\epsilon_0}(V^k,U^k) + I^k_{\Phi_t}+I^k_{\Phi_x}+I^k_{\Phi}+I^k_0.
			\end{align}
			Since $U^k(t,x)=\frac{\overline{U}^k(t,x)}{ \sigma^k(x)},$  and $|a-b|-|c-d|\le |a-c|+|b-d|,$ we have,
			\begin{align*}
				I^k_{\Phi_t} 
				&\leq \int_{Q_T^2} 
				\left( \abs{\overline{U}^k(t,x)} + \abs{\overline{V}^k(s,y)}\right)\left|\frac{1}{ \sigma^k(x)}-\frac{1}{ \sigma^k(y)}\right|\abs{\omega_{\epsilon_0}'(t-s)} \omega_{\epsilon}(x-y) \d t \d x \d s \d y \\ &\leq T\left(\norma{\overline{U}}_{(L^{\infty}(Q_T))^N}+ \norma{\overline{V}}_{(L^{\infty}(Q_T))^N}\right) \frac{\epsilon}{\epsilon_0} \left|\frac{1}{ s}\right|_{(\bv(\R))^N}.
			\end{align*}
			The terms $I^k_0$ and $I^k_T$ appear for local conservation laws with Panov type flux as well and can be estimated as in \cite[Lemma~2]{GTV2022} to get:
			\begin{align*}
				I^k_T&\ge \norma{U^k(T,\dott)-V^k(T,\dott)}_{L^1(\R)}-\mathcal{C}_1(\epsilon+\epsilon_0+\gamma(V^k,\epsilon_0)),\\
				I^k_0&\le \norma{U^k_0-V^k_0}_{L^1(\R)}+\mathcal{C}_1(\epsilon+\epsilon_0+\gamma(V^k,\epsilon_0)),
			\end{align*}
			where $\mathcal{C}_1=\mathcal{C}_1(|\boldsymbol{U}|_{
				(L^{\infty}([0,T]; \bv(\R)))^{N}}, |\boldsymbol{V}|_{(L^{\infty}([0,T]; \bv(\R)))^{N}},|\boldsymbol{U}|_{\lip([0,T]; L^1(\R))^{N}}).$ It is to be noted that $I^k_{\Phi_x}$  and $I^k_{\Phi}$ comprise of the coupled terms from the system, which we now estimate. Using integration by parts we have, 
			\begin{align*}
				I^k_{\Phi_x}&+I^k_{\Phi}\\=&-\int_{Q^2_T}\Phi G^k_x(\overline{U}^k(t,x),\overline{V}^k(s,y))(\mathcal{U}^k(t,x)-\mathcal{V}^k(s,y))\d t \d x \d s \d y\\
				&-\int_{Q^2_T}\Phi \sgn (\overline{U}^k(t,x)-\overline{V}^k(s,y)) f^k(\overline{U}^k(t,x))\mathcal{U}^k_x(t,x)\d t \d x \d s \d y\\
				&+\int_{Q^2_T}\Phi\sgn (\overline{U}^k(t,x)-\overline{V}^k(s,y)) f^k(\overline{U}^k(t,x))  \mathcal{V}^k_y(s,y)\d x dt \d s \d y \\:=&I^k_\mathcal{U}+I^k_{\mathcal{U}_x}.
			\end{align*}
			Let $\abs{\overline{U}^k_x}$ denote the absolute value of the Radon measure $\overline{U}^k_x$ which is finite as it is a distributional derivative of a $\bv$ function. Owing to the Lipschitz continuity of the entropy flux $G,$ we have
			$|\partial_xG^k(\overline{U}^k(t,x),\overline{V}^k(s,y))|\le \abs{f^k}_{\lipR}\abs{\overline{U}^k_x}$ (see~[Lemma~A2.1]\cite{BP1998} for details), implying
			\begin{align*}
				I_{\mathcal{U}^k}&\le \abs{f^k}_{\lipR}\int_{Q^2_T}\Phi\abs{\overline{U}^k_x}\abs{\mathcal{V}^k(s,y)-\mathcal{U}^k(t,x)}\d{x} \d{y}\d{t} \d{s}.
			\end{align*}
			Further, the nonlocal weights $\mathcal{V}^k(s,y),\mathcal{U}^k(t,x)$ satisfy, 
			\begin{align}
				&\abs{\mathcal{V}^k(s,y)-\mathcal{U}^k(t,x)}\nonumber\\
				&\quad \leq\nonumber\abs{\mathcal{V}^k(s,y)-\mathcal{V}^k(s,x)}+\abs{\mathcal{V}^k(s,x)-\mathcal{U}^k(t,x)}\nonumber\\
				&\quad=\nonumber|\nu^k((\boldsymbol{\mu}*\boldsymbol{V})^k(s,y))-\nu^k((\boldsymbol{\mu}*\boldsymbol{V})^k(s,x))|\\
				&\qquad+|\nu^k((\boldsymbol{\mu}*\boldsymbol{V})^k(s,x))-\nu^k((\boldsymbol{\mu}*\boldsymbol{U})^k(t,x))|\nonumber\\&\quad \leq \abs{\nu^k}_{\lipR}\norma{(\boldsymbol{\mu}*\boldsymbol{V})^k(s,y)-(\boldsymbol{\mu}*\boldsymbol{V})^k(s,x)}\nonumber\\&\qquad+\abs{\nu^k}_{\lipR}\norma{(\boldsymbol{\mu}*\boldsymbol{V})^k(s,x)-(\boldsymbol{\mu}*\boldsymbol{U})^k(t,x)}\nonumber\\
				&\quad=\abs{\nu^k}_{\lipR}\norma{\boldsymbol{\mu}'}_{(L^{\infty}(\R))^{N^2}}\norma{\boldsymbol{V}(s,\dott)}_{(L^1(\R))^N}\abs{y-x}\nonumber\\&\qquad+\abs{\nu^k}_{\lipR} \norma{\boldsymbol{\mu}}_{(L^{\infty}(\R))^{N^2}}
				\norma{\boldsymbol{V}(s,\dott)-\boldsymbol{U}(t,\dott)}_{(L^1(\R))^N}\nonumber. 
			\end{align}
			It can been seen that the estimates on nonlocal weights for any $k,$ depend on all the $N$ components of $\boldsymbol{U}$ and $\boldsymbol{V}.$
			Consequently, we get:
			\begin{align*}
				I_{\mathcal{U}^k}&\leq\abs{f^k}_{\lipR}\abs{\nu^k}_{\lipR}\norma{\boldsymbol{\mu}'}_{(L^{\infty}(\R))^{N^2}}\\&\qquad \times \int_{Q^2_T}\Phi\abs{\overline{U}^k_x}\norma{\boldsymbol{V}(s,\dott)}_{(L^1(\R))^N}\abs{y-x}\d{x} \d{y}\d{t} \d{s}\\
				&\quad +\abs{f^k}_{\lipR}\abs{\nu^k}_{\lipR}\norma{\boldsymbol{\mu}}_{(L^{\infty}(\R))^{N^2}}\\&\qquad\times\int_{Q^2_T}\Phi\abs{\overline{U}^k_x}
				\norma{\boldsymbol{V}(s,\dott)-\boldsymbol{U}(t,\dott)}_{(L^1(\R))^N}\d{x} \d{y}\d{t} \d{s}\\
				&\leq
				\abs{f^k}_{\lipR}\abs{\nu^k}_{\lipR}\norma{\boldsymbol{\mu}'}_{(L^{\infty}(\R))^{N^2}}\\&\qquad \times \int_{Q^2_T}\Phi\abs{\overline{U}^k_x}\norma{\boldsymbol{V}(s,\dott)}_{(L^1(\R))^N}\abs{y-x}\d{x} \d{y}\d{t} \d{s}\\
				&\quad+\abs{f^k}_{\lipR}\abs{\nu^k}_{\lipR}\norma{\boldsymbol{\mu}}_{(L^{\infty}(\R))^{N^2}}\\&\qquad\quad\times\int_{Q^2_T}\Phi\abs{\overline{U}^k_x}
				\norma{\boldsymbol{U}(s,\dott)-\boldsymbol{U}(t,\dott)}_{(L^1(\R))^N}\d{x} \d{y}\d{t}\d{s}\\&\quad+\abs{f^k}_{\lipR}\abs{\nu^k}_{\lipR} \norma{\boldsymbol{\mu}}_{(L^{\infty}(\R))^{N^2}}\\&\qquad\quad\times\int_{Q^2_T}\Phi\abs{\overline{U}^k_x}
				\norma{\boldsymbol{V}(s,\dott)-\boldsymbol{U}(s,\dott)}_{(L^1(\R))^N}\d{x} \d{y}\d{t} \d{s},
				\\&:= I^1_{\mathcal{U}^k}+I^2_{\mathcal{U}^k}+I^3_{\mathcal{U}^k},
			\end{align*}
			where
			\begin{align*}
				I^1_{\mathcal{U}^k
				}
				&\le\abs{f^k}_{\lipR}\abs{\nu^k}_{\lipR}\norma{\boldsymbol{\mu}'}_{(L^{\infty}(\R))^{N^2}}\\&\qquad \times  \int_{Q_T}\int_{Q_T}\omega_{\epsilon}(x-y)\omega_{{\epsilon}_0}(t-s)\abs{\overline{U}^k_x}\norma{\boldsymbol{V}(s,\dott)}_{(L^1(\R))^N}\epsilon\d{x}
				\d{y}\d{t}\d{s}\\
				&\leq\abs{f^k}_{\lipR}\abs{\nu^k}_{\lipR}\norma{\boldsymbol{\mu}'}_{(L^{\infty}(\R))^{N^2}}\abs{\overline{U}^k}_{(L^{\infty}([0,T];\bv(\R)))}\norma{\boldsymbol{V}}_{(L^1(Q_T))^N}\epsilon,
				\\[2mm]
				I^2_{\mathcal{U}^k}
				&\le\abs{f^k}_{\lipR}\abs{\nu^k}_{\lipR} |\boldsymbol{U}|_{(\lip(Q_T))^N}\norma{\boldsymbol{\mu}}_{(L^{\infty}(\R))^{N^2}}\\&\qquad \times  \int_{Q^2_T}\omega_{\epsilon}(x-y)\omega_{{\epsilon}_0}(t-s)\abs{\overline{U}^k_x}
				|t-s|\d{x} \d{y}\d{t}\d{s}\\
				&\le \abs{f^k}_{\lipR}\abs{\nu^k}_{\lipR}\abs{\overline{U}^k}_{(L^{\infty}([0,T];\bv(\R)))}|\boldsymbol{U}|_{(\lip(Q_T))^N} \norma{\boldsymbol{\mu}}_{(L^{\infty}(\R))^{N^2}}\\&\qquad \times \int_0^T\int_0^T\omega_{{\epsilon}_0}(t-s)
				{\epsilon_0}\d{t}\d{s}\\
				&=\abs{f^k}_{\lipR}\abs{\nu^k}_{\lipR}\abs{\overline{U}^k}_{(L^{\infty}([0,T];\bv(\R)))}|\boldsymbol{U}|_{(\lip(Q_T))^N}\norma{\boldsymbol{\mu}}_{(L^{\infty}(\R))^{N^2}} T
				{\epsilon_0},
				\\[2mm]
				I^3_{\mathcal{U}^k}
				&\leq\abs{f^k}_{\lipR}\abs{\nu^k}_{\lipR}\abs{\overline{U}^k}_{(L^{\infty}([0,T];\bv(\R)))}\norma{\boldsymbol{\mu}}_{(L^{\infty}(\R))^{N^2}}\\&\qquad \times \int_{0}^T
				\norma{\boldsymbol{V}(s,\dott)-\boldsymbol{U}(s,\dott)}_{(L^1(\R))^N}\d{s}.
			\end{align*}
			Collectively we get,
			\begin{align*}
				I_{\mathcal{U}^k}&\le \abs{f^k}_{\lipR}\abs{\nu^k}_{\lipR}\abs{\overline{U}^k}_{(L^{\infty}([0,T];\bv(\R)))}\norma{\boldsymbol{\mu}'}_{(L^{\infty}(\R))^{N^2}}\norma{\boldsymbol{V}}_{(L^1(Q_T))^N}\epsilon\\
				&\quad +\abs{f^k}_{\lipR}\abs{\nu^k}_{\lipR}\abs{\overline{U}^k}_{(L^{\infty}([0,T];\bv(\R)))}\norma{\boldsymbol{\mu}}_{(L^{\infty}(\R))^{N^2}}\\&\qquad \times \left(|\boldsymbol{U}|_{(\lip(Q_T))^N}T
				\epsilon_0+\int_{0}^T
				\norma{\boldsymbol{V}(s,\dott)-\boldsymbol{U}(s,\dott)}_{(L^1(\R))^N}\d{s}\right)\\
				&\leq \mathcal{K}_1(\epsilon+\epsilon_0)+K_2 \int_{0}^T
				\norma{\boldsymbol{V}(s,\dott)-\boldsymbol{U}(s,\dott)}_{(L^1(\R))^N}\d{s}.
			\end{align*}
			Similarly, we can show that the following estimate holds.
			\begin{align*}
				I_{\mathcal{U}^k_x}&\leq \mathcal{K}_3(\epsilon+\epsilon_0)+K_4 \int_{0}^T
				\norma{\boldsymbol{V}(s,\dott)-\boldsymbol{U}(s,\dott)}_{(L^1(\R))^N}\d{s}.
			\end{align*}
			Substituting the above estimates in \eqref{kuz} and summing over $k$, we get
			\begin{align*}
				&\norma{\boldsymbol{U}(T)-\boldsymbol{V}(T)}_{(L^1(\R))^N}\\ &\qquad \le\norma{\boldsymbol{V}_0-\boldsymbol{U}_0}_{(L^1(\R))^N}+\mathcal{C}_2\int_{0}^T
				\norma{\boldsymbol{V}(t,\dott)-\boldsymbol{U}(t,\dott)}_{(L^1(\R))^N}\d{t}\\
				&\qquad \quad +\mathcal{C}_1N\left(\epsilon+\epsilon_0+ \frac{\epsilon}{\epsilon_0}\abs{\frac{1}{ s}}_{(\bv(\R))^N}\right)+\mathcal{C}_1\sum\limits_{k=1}^N\left(\gamma(V^k,\epsilon_0)-\Lambda^k_{\epsilon,\epsilon_0}(V^k,U^k)\right).\end{align*}
			Now, the lemma follows by invoking Gronwall's inequality.
		\end{proof}
		\begin{theorem}\label{uniqueness}[Uniqueness of the entropy solution for nonlocal systems]
			For any time $T>0,$ let $\boldsymbol{U,V}$ 
			be the entropy solutions of the IVP for the system \eqref{eq:uA}
			with initial data $\boldsymbol{U}_0,\boldsymbol{V}_0$, respectively. Then, the following holds:
			\begin{align*}
				\norma{\boldsymbol{U}(T,\dott)-\boldsymbol{V}(T,\dott)}_{(L^1(\R))^N} &\leq \mathcal{C}\norma{\boldsymbol{U}_0-\boldsymbol{V}_0}_{(L^1(\R))^N},
			\end{align*}
			where $\mathcal{C}$ is the same as in Lemma~\ref{lemma:kuz}.
			In particular, if $\boldsymbol{U}_0=\boldsymbol{V}_0$ then $\boldsymbol{U}=\boldsymbol{V}$ a.e. in $\overline{Q}_T.$
		\end{theorem}
		\begin{proof}
			Since $\boldsymbol{U,V}$ are the entropy solutions, we have $-\sum_{k=1}^N\Lambda^k_{\epsilon,\epsilon_0}(V^k,U^k) \leq 0,$ for all $\epsilon,\epsilon_0>0.$
			Now, the proof follows from Lemma~\ref{lemma:kuz} by sending $\epsilon \rightarrow 0^+$ and then $\epsilon_0 \rightarrow 0^+.$
		\end{proof}
		\begin{remark}\normalfont
			Theorem~\ref{uniqueness}, in combination with Theorem~\ref{Existence} (2), 
			implies that in fact the entire sequence $\boldsymbol{U}_{\D}$ of the finite volume approximations converge to the unique entropy solution $\boldsymbol{U}$ in $(L_{\rm loc}^1(\overline{Q}_T))^N$ and pointwise a.e. in $\overline{Q}_T.$
		\end{remark}
		In the next section, we establish the convergence rates for the proposed numerical schemes.

		%------------------ section
		\section{Error Estimate}\label{sec:error}
		For every $n  \in  \{0, \ldots, N_T\},~  i\in \Z,~ 
		k  \in  \{1, \ldots, N\}$, $ (t,x)\in \overline{Q}_T$ and $\alpha\in\R,$ we define the following:
		\begin{enumerate}[i.]
			\item $\eta_{i}^{k,n}(\alpha):=\left|U_i^{k,n}- \frac{\alpha}{\sigma_i^k}\right|$,
			\item $p_i^{k,n}(\alpha):=\sgn (\overline{U}_i^{k,n}-\alpha) (f^k(\overline{U}_i^{k,n})-f^k(\alpha)) $,
			\item $\mathcal{U}^{k}_{\Delta}(t,x):= \nu^k((\boldsymbol{\mu}*\boldsymbol{U}_{\Delta})^k(t,x))$.
		\end{enumerate}
		It is to be noted that owing to the discontinuous flux, $\Lambda^k_{\epsilon,\epsilon_0}$ has non-piecewise constant terms arising from $\boldsymbol{\sigma}(x)$, unlike the classical cases considered in \cite{HR2015}. In comparison to a similar case considered in \cite{GTV2022a}, the PDE \eqref{eq:uA}--\eqref{eq:u11A} consists of nonlocal coupled terms in addition, making it non-trivial. To address this issue, we first prove an estimate on $\Lambda^k_{\Delta,\epsilon,\epsilon_0}$, an appropriate approximation of $\Lambda^k_{\epsilon,\epsilon_0},$ (with $\boldsymbol{\sigma}_{\D}$ in place of $\boldsymbol{\sigma}$) in Lemma~\ref{lemma:est} and then estimate the difference between $\Lambda^k_{\epsilon,\epsilon_0}$  and $\Lambda^k_{\Delta,\epsilon,\epsilon_0}.$ If the local part of the flux is homogeneous (i.e.,~$\TV( s)=0$), then this step is not required and in fact, $\Lambda^k_{\epsilon,\epsilon_0}=\Lambda^k_{\Delta,\epsilon,\epsilon_0}.$ 
		\begin{lemma}\label{lemma:est} For every $k\in\{1,\ldots,N\},$ the relative entropy functional \\$\Lambda^k_{\epsilon,\epsilon_0} (U^{k}_{\Delta},U^k)$ satisfies:
			\begin{align}
				\nonumber-\Lambda^k_{\epsilon,\epsilon_0} (U^{k}_{\Delta},U^k)&\leq \mathcal{M}\left(\frac{\D x}{\epsilon}+\frac{\D t}{\epsilon_0}\right),
			\end{align}
			where $\mathcal{M}$ is a constant independent of $\Delta x, \Delta t.$
		\end{lemma}
		\begin{proof}For any $k\in\{1,\ldots,N\}$, we define
			\begin{align*}
				&-\Lambda^k_{\Delta,\epsilon,\epsilon_0}(U^{k}_{\Delta},U^k)\\&\quad:= -\int_{Q_T}\sum_{i\in \Z}\sum\limits_{n=0}^{N_T-1}   \int_{C_i^n}\eta_{i}^{k,n}(\overline{U}^k(s,y))\Phi_t(s,y,t,x) \d t  \d x \d s \d y \nonumber\\
				&\qquad-\int_{Q_T}\sum_{i\in \Z}\sum\limits_{n=0}^{N_T-1}   \int_{C_i^n} p_i^{k,n}(\overline{U}^k(s,y))\mathcal{U}^{k}_{\Delta}(t,x)\Phi_x(s,y,t,x) \d t  \d x \d s \d y \nonumber\\
				&\qquad + \int_{Q_T}\sum_{i\in \Z}\sum\limits_{n=0}^{N_T-1} \int_{C_i^n}\sgn\left(U_i^{k,n}-\frac{\overline{U}^k(s,y)}{ \sigma^k_i}\right) \\&\qquad \qquad\qquad\qquad\quad\times f^k(\overline{U}^k(s,y))\partial_x\mathcal{U}^{k}_{\Delta}(t,x)\Phi(s,y,t,x) \d t  \d x \d s \d y \\
				&\qquad- \int_{Q_T}\sum_{i}   \int_{C_i}  \eta_{i}^{k,0}(\overline{U}^k(s,y)) \Phi(s,y,0,x)  \d x\d s  \d y\\&\qquad +\int_{Q_T} \sum_{i}  \int_{C_i}  \eta_{i}^{k,N_T}(\overline{U}^k(s,y))  \Phi(s,y,T,x)  \d x\d s  \d y.\end{align*}
			Consequently,
			\begin{align*}&\abs{\Lambda^k_{\Delta,\epsilon,\epsilon_0} (U^{k}_{\Delta},U^k)-\Lambda^k_{\epsilon,\epsilon_0} (U^{k}_{\Delta},U^k)}\\ &\quad\leq C \int_{Q_T^2}  \left|\frac{\alpha}{\sigma^k(x)}-\frac{\alpha}{\sigma^k_{\D}(x)}\right| (\abs{\Phi_t}+ \norma{\mathcal{U}}_{L^{\infty}(Q_T)} \abs{\Phi_x}) \d x\d t  \d s \d y \\&\qquad + C \int_{Q_T^2}  \left|\frac{\alpha}{\sigma^k(x)}-\frac{\alpha}{\sigma^k_{\D}(x)}\right|\norma{\mathcal{U}_x}_{L^{\infty}(Q_T)}\abs{\Phi}\d x\d t  \d s \d y \\&\qquad+ C\int_{ Q_T \times \R} \left| \frac{\alpha}{\sigma^k(x)}-\frac{\alpha}{\sigma^k_{\D}(x)}\right|\left( \abs{\Phi(\dott,\dott,\dott,T) } + 
				\abs{\Phi(\dott,\dott,\dott,0)}\right) \d x\d s  \d y\\ &\quad \leq C \left( \D x + \frac{\D x}{\epsilon}+ \frac{\D x}{\epsilon_0}\right).
			\end{align*}
			In what follows, we estimate $\Lambda^k_{\Delta,\epsilon,\epsilon_0} (U^{k}_{\Delta},U^k),$ which is enough to get the desired result. 
			To begin with, we invoke fundamental theorem of calculus followed by summation by parts to get, 
			\begin{align*}
				&\Lambda^k_{\Delta,\epsilon,\epsilon_0} (U^{k}_{\Delta},U^k)\\&\quad=\int_{Q_T}\sum_{i\in \Z}\sum\limits_{n=0}^{N_T-1} \left(\eta_i^{k,n+1}(\overline{U}^k(s,y))-\eta_i^{k,n}(\overline{U}^k(s,y))\right)\int_{C_i}\Phi(s,y,t^{n+1},x)  \d x\d s  \d y\\
				&\qquad-\int_{Q_T}\sum_{i\in \Z}\sum\limits_{n=0}^{N_T-1}   \int_{C_i^n}p_i^{k,n}(\overline{U}^k(s,y))\mathcal{U}^{k}_{\Delta}(t,x)\Phi_x(s,y,t,x) \d t  \d x \d s \d y  \nonumber\\
				&\qquad + \int_{Q_T}\sum_{i\in \Z}\sum\limits_{n=0}^{N_T-1}\int_{C_i^n}\sgn\left(U_i^{k,n}-\frac{\overline{U}^k(s,y)}{ \sigma^k_i}\right)    \\&\qquad \qquad\qquad\qquad\quad\times f^k(\overline{U}^k(s,y))\partial_x\mathcal{U}^{k}_{\Delta}(t,x)\Phi(s,y,t,x) \d t  \d x \d s \d y 
				\\&\quad:=\lambda_1+\lambda_2+\lambda_3.
			\end{align*}
			Now, we make the following estimates:
			\begin{align*}
				\lambda_2
				&=-\int_{Q_T}\sum_{i\in \Z}\sum\limits_{n=0}^{N_T-1}  \int_{C_i^n}p_i^{k,n}(\overline{U}^k(s,y))\nu^k({\boldsymbol{c}}_{i+1/2}^{k,n})\Phi_x(s,y,t,x)  \d t\d x\d s  \d y\\ 
				&\quad-\int_{Q_T}\sum_{i\in \Z}\sum\limits_{n=0}^{N_T-1} \int_{C_i^n}p_i^{k,n}(\overline{U}^k(s,y)) \\&\qquad \qquad\qquad\qquad\quad\times   \Phi_x(s,y,t,x) \left(\mathcal{U}^{k}_{\Delta}(t,x)-\nu^k({\boldsymbol{c}}_{i+1/2}^{k,n})\right)  \d t\d x\d s  \d y\\
				&:=\lambda_{2'}+ \mathcal{E}_2.
			\end{align*}
			For $(t,x)\in C_i^n,$ we have 
			\begin{align*}
				&\left|\mathcal{U}^{k}_{\Delta}(t,x)-\nu^k({\boldsymbol{c}}_{i+1/2}^{k,n})\right|\\&\quad
				\le \left|\mathcal{U}^{k}_{\Delta}(t,x)-\mathcal{U}^{k}_{\Delta}(t^n,x)\right| +\left|\mathcal{U}^{k}_{\Delta}(t^n,x)-\mathcal{U}^{k}_{\Delta}(t^n,x_{i+1/2})\right|\\&\qquad +\left|\mathcal{U}^{k}_{\Delta}(t^n,x_{i+1/2})-\nu^k({\boldsymbol{c}}_{i+1/2}^{k,n})\right|.
			\end{align*}
			Furthermore,  
			\begin{align}  
				\nonumber
				&\abs{\mathcal{U}^{k}_{\Delta}(t^n,x_{i+1/2})-\nu^k({\boldsymbol{c}}_{i+1/2}^{k,n})}\\
				&\quad\le \Delta x\abs{\nu^k}_{\lipR}\norma{\boldsymbol{\mu}'}_{(L^{\infty}(\R))^{N^2}}\norma{\boldsymbol{U}_{\Delta}}_{(L^1(Q_T))^N} :=\mathcal{M}_1\Delta x\label{est:Uminusc}.
			\end{align}
			Now using the Lipschitz continuity of $\mathcal{U}^{k}_{\Delta}$ in the space variable, we have
			\begin{align*}
				&\abs{\mathcal{U}^{k}_{\Delta}(t^n,x)-\mathcal{U}^{k}_{\Delta}(t^n,x_{i+1/2})}\leq \mathcal{M}_2\Delta x.
			\end{align*}
			Also for $t \in C^n,$ using \eqref{apx:time}, we have\begin{align*}
				\abs{\mathcal{U}^{k}_{\Delta}(t^n,x)-\mathcal{U}^{k}_{\Delta}(t,x)}
				&\le \abs{\nu^k}_{\lipR}\norma{\boldsymbol{\mu}}_{(\L\infty(\R))^{N^2}}\mathcal{K}_6\Delta t:=\mathcal{M}_3\Delta x. 
			\end{align*}
			Finally, combining all the above estimates, we get, \begin{align*}
				\abs{\mathcal{E}_2}
				&\leq\mathcal{M}_4\Delta x \int_{Q_T}\sum_{i\in \Z}\sum\limits_{n=0}^{N_T-1}  \left(\abs{f^k(\overline{U}_i^{k,n})}+\abs{f^k(\overline{U}^k(s,y)}\right)\\&\qquad \qquad\qquad\qquad\qquad\quad\times \int_{C_i^n} \modulo{\Phi_x(s,y,t,x)}   \d t\d x\d s  \d y\\
				&=\mathcal{M}_4\Delta x \sum_{i\in \Z}\sum\limits_{n=0}^{N_T-1}\int_{C_i^n}   \abs{f^k(\overline{U}_i^{k,n})}\int_{Q_T}\abs{\Phi_x(s,y,t,x)}   \d t\d x\d s  \d y\\
				&\qquad+\mathcal{M}_4\Delta x\int_{Q_T}\int_{Q_T}  \abs{f^k(\overline{U}^k(s,y)\Phi_x(s,y,t,x)  } \d t\d x\d s  \d y\\
				&=\tilde{\mathcal{M}_4}\abs{f^k}_{\lipR}\norma{\overline{U}^{k}_{\D}}_{L^1(Q_T)}\frac{\Delta x}{\epsilon}+  \tilde{\mathcal{M}_4}\abs{f^k}_{\lipR}\norma{\overline{U}^k}_{L^1(Q_T)}\frac{\Delta x}{\epsilon}.
			\end{align*}
			Thus, we have
			\begin{equation*}
				\lambda_2= \lambda_2'+\mathcal{O}\left(\frac{\Delta x}{\epsilon}\right).
			\end{equation*}
			Now, we consider
			\begin{align*}
				\lambda_3
				&= \int_{Q_T}\sum_{i\in \Z}\sum\limits_{n=0}^{N_T-1}  \sgn\left(U_i^{k,n}-\frac{\overline{U}^k(s,y)}{ \sigma^k_i}\right)(\nu^k({\boldsymbol{c}}_{i+1/2}^{k,n})-\nu^k({\boldsymbol{c}}_{i-1/2}^{k,n}))\\
				\\&\qquad \qquad\qquad\qquad\times  f^k(\overline{U}^k(s,y)) \int_{C^{n}} \Phi(s,y,t,x_{i+1/2})  \d t \d s \d y \\
				&\qquad+ \int_{Q_T}\sum_{i\in \Z}\sum\limits_{n=0}^{N_T-1}  \sgn\left(U_i^{k,n}-\frac{\overline{U}^k(s,y)}{ \sigma^k_i}\right)
				f^k(\overline{U}^k(s,y)) \int_{C_i^n}\mathcal{U}^{k}_{\Delta,x}(t,x)\\
				\\&\qquad \qquad\qquad\qquad\qquad\times  
				\left(\Phi(s,y,t,x)-\Phi(s,y,t,x_{i+1/2})\right) \d x \d s \d t \d y\\
				&\qquad+ \int_{Q_T}\sum_{i\in \Z}\sum\limits_{n=0}^{N_T-1}  \sgn\left(U_i^{k,n}-\frac{\overline{U}^k(s,y)}{ \sigma^k_i}\right) f^k(\overline{U}^k(s,y))\int_{C_i^n}\Phi(s,y,t,x_{i+1/2})  \\
				\\&\qquad \qquad\qquad\qquad\qquad\times 
				\left(\mathcal{U}^{k}_{\Delta,x}(t,x)  -\frac{\nu^k({\boldsymbol{c}}_{i+1/2}^{k,n})-\nu^k({\boldsymbol{c}}_{i-1/2}^{k,n})}{\Delta x}\right)\d x \d s \d t \d y\\
				&:=\lambda_3'+\mathcal{E}_{31}+\mathcal{E}_{32}.
			\end{align*}
			Further,
			\begin{align*}
				\abs{\mathcal{E}_{31}} &\leq  \int_{Q_T}\sum_{i\in \Z}\sum\limits_{n=0}^{N_T-1}  \abs{f^k(\overline{U}^k(s,y)) }
				\\&\qquad \qquad\times
				\int_{C_i^n} \abs{\mathcal{U}^{k}_{\Delta,x}(t,x) }  \abs{ \Phi(s,y,t,x)-\Phi(s,y,t,x_{i+1/2})} \d t \d x \d s \d y \\ &\quad\leq 
				|\nu^k|_{\lipR}\norma{\boldsymbol{\mu'}}_{(L^\infty(\R))^{N^2}} \norma{\boldsymbol{U}_{0}}_{(L^1(\R))^N}
				\int_{Q_T}\sum_{i\in \Z}\sum\limits_{n=0}^{N_T-1}  \abs{f^k(\overline{U}^k(s,y)) }\\
				\\&\qquad \qquad\times
				\int_{C_i^n}   \abs{ \Phi(s,y,t,x)-\Phi(s,y,t,x_{i+1/2})} \d x \d s \d t \d y\\&\quad\leq  |\nu^k|_{\lipR}\norma{\boldsymbol{\mu'}}_{(L^\infty(\R))^{N^2}} \norma{\boldsymbol{U}_{0}}_{(L^1(\R))^N}
				\int_{Q_T} \abs{f^k(\overline{U}^k(s,y)) }
				\\&\qquad \qquad\times
				\sum_{i}  \int_{C_i}     \abs{\omega_{\epsilon} (y-x)-\omega_{\epsilon}(y-x_{i+1/2})} \d x \d s  \d y
				\\&\quad\leq  |\nu^k|_{\lipR}\norma{\boldsymbol{\mu'}}_{(L^\infty(\R))^{N^2}} \norma{\boldsymbol{U}_{0}}_{(L^1(\R))^N}|\omega_{\epsilon}|_{BV(\R)} \D x \int_{Q_T} \abs{f^k(\overline{U}^k(s,y)) }
				\d s  \d y \\
				&\quad\leq |\nu^k|_{\lipR}\norma{\boldsymbol{\mu'}}_{(L^\infty(\R))^{N^2}} \norma{\boldsymbol{U}_{0}}_{(L^1(\R))^N}\norma{f^k(\overline{U}^k)}_{L^1(Q_T)}\frac{\D x}{\epsilon},
			\end{align*}
			since 
			\begin{align*}
				|\mathcal{U}^{k}_{\Delta,x}(t,x)|&=|\nu^{k}_x((\boldsymbol{\mu}*\boldsymbol{U})^k)(t,x)(\boldsymbol{\mu'}*\boldsymbol{U})^k(t,x)|\\
				&\leq |\nu^k|_{\lipR}\norma{\boldsymbol{\mu'}}_{(L^\infty(\R))^{N^2}} \norma{\boldsymbol{U}_{0}}_{(L^1(\R))^N}.
			\end{align*}
			Next, we consider
			\begin{align*}
				\abs{\mathcal{E}_{32} }
				&=\big|\int_{Q_T}\sum_{i\in \Z}\sum\limits_{n=0}^{N_T-1}  f^k(\overline{U}^k(s,y)) \int_{C^n}\Phi(s,y,t,x_{i+1/2})
				\\&\qquad \times \int_{C_i}  \partial_x\mathcal{U}^{k}_{\Delta}(t,x)-\frac{\nu^k({\boldsymbol{c}}_{i+1/2}^{k,n})-\nu^k({\boldsymbol{c}}_{i-1/2}^{k,n})}{\Delta x}\d x \d s \d t \d y\big|. 
			\end{align*}
			Applying
			the fundamental theorem of calculus and rearranging the terms, we get
			\begin{align*}
				\abs{\mathcal{E}_{32}}&=\Big|\int_{Q_T}\sum_{i\in \Z}\sum\limits_{n=0}^{N_T-1}  \left(
				\mathcal{U}^{k}_{\Delta}(t^n,x_{i+1/2})- \mathcal{U}^{k}_{\Delta}(t^n,x_{i-1/2})-\nu^k({\boldsymbol{c}}_{i+1/2}^{k,n})+\nu^k({\boldsymbol{c}}_{i-1/2}^{k,n})\right)  \\&\qquad\qquad\qquad\qquad\times{\int_{C^n}f^k(\overline{U}^k(s,y))\Phi(s,y,t,x_{i+1/2}) \d s \d t \d y} 
				\\
				&\quad+{\int_{Q_T}\sum_{i\in \Z}\sum\limits_{n=0}^{N_T-1}  f^k(\overline{U}^k(s,y)) \int_{C^n}\Phi(s,y,t,x_{i+1/2})} \\&\qquad\qquad\qquad\qquad\qquad\qquad\qquad\times\big(
				(\mathcal{U}^{k}_{\Delta}(t,x_{i+1/2})- \mathcal{U}^{k}_{\Delta}(t^n,x_{i+1/2})) -\\&\qquad\qquad\qquad\quad\qquad\qquad\qquad\qquad-( \mathcal{U}^{k}_{\Delta}(t,x_{i-1/2})-\mathcal{U}^{k}_{\Delta}(t^n,x_{i-1/2}))\big) \d s \d t \d y \Big|.
			\end{align*}
			Applying summation by parts in $i$ we get
			\begin{align*}
				\abs{\mathcal{E}_{32} }
				&\le\Big|{\int_{Q_T}\sum_{i\in \Z}\sum\limits_{n=0}^{N_T-1}  } \left(
				\mathcal{U}^{k}_{\Delta}(t^n,x_{i+1/2})-\nu^k({\boldsymbol{c}}_{i+1/2}^{k,n})\right) \\&\qquad\qquad\quad\qquad\times\int_{C^n}f^k(\overline{U}^k(s,y)) (\Phi(s,y,t,x_{i+1/2})-\Phi(s,y,t,x_{i-1/2})) \d s \d t \d y  \Big|
				\\
				&\qquad  +\Big|{\int_{Q_T}\sum_{i\in \Z}\sum\limits_{n=0}^{N_T-1}  } \int_{C^n} f^k(\overline{U}^k(s,y))\big(
				\mathcal{U}^{k}_{\Delta}(t,x_{i+1/2})- \mathcal{U}^{k}_{\Delta}(t^n,x_{i+1/2})\big) \\
				& \qquad\qquad\qquad\qquad \quad \times(\Phi(s,y,t,x_{i+1/2})-\Phi(s,y,t,x_{i-1/2})\d s \d t \d y\big|\\
				&:= \mathcal{E}_{321}+ \mathcal{E}_{322}.
			\end{align*}
			Now,  using \eqref{est:Uminusc}, we have
			\begin{align*}
				\mathcal{E}_{321}
				&\leq\int_{Q_T} \abs{f^k(\overline{U}^k(s,y)) } \abs{\nu^k}_{\lipR}\norma{U^{k}_{\Delta}}_{L^1(Q_T)} \norma{\boldsymbol{\mu}'}_{(L^{\infty}(\R))^{N^2}}\Delta x|\omega_{\epsilon}|_{BV(\R)}\\
				& \qquad  \times \int_{0}^T|\omega_{\epsilon_0}(s-t)| \d s \d t \d y
				\\
				&\le {\int_{Q_T} \abs{f^k(\overline{U}^k(s,y)) } \abs{\nu^k}_{\lipR}\norma{U^{k}_{\Delta}}_{L^1(Q_T)} \norma{\boldsymbol{\mu}'}_{(L^{\infty}(\R))^{N^2}} \frac{\Delta x}{\epsilon} \d s  \d y}\\
				&= {\abs{\nu^k}_{\lipR}\norma{U^{k}_{\Delta}}_{L^1(Q_T)} \norma{\boldsymbol{\mu}'}_{(L^{\infty}(\R))^{N^2}}\norma{f^k(\overline{U}^k)}_{L^1(Q_T)} \frac{\Delta x}{\epsilon}}.
			\end{align*}
			Using \eqref{apx:time}, for $t\in C^n,$ we have,
			\begin{align*}
				&\abs{ \mathcal{U}^{k}_{\Delta}(t,x_{i+1/2})- \mathcal{U}^{k}_{\Delta}(t^n,x_{i+1/2})} \\
				&\quad\leq \abs{\nu^k}_{\lipR}\norma{\boldsymbol{\mu}}_{(L^{\infty}(\R))^{N^2}} \norma{\boldsymbol{U}_{\D}(t,\dott)-\boldsymbol{U}_{\D}(t^n,\dott)}_{(L^1(\R))^N}\\
				&\quad\le \abs{\nu^k}_{\lipR}\norma{\boldsymbol{\mu}}_{(L^{\infty}(\R))^{N^2}} \mathcal{K}_6\Delta t.
			\end{align*}
			
			Consequently,
			\begin{align*}
				\mathcal{E}_{322}
				&\leq\int_{Q_T} \abs{f^k(\overline{U}^k(s,y)) } \abs{\nu^k}_{\lipR}\norma{\boldsymbol{\mu}}_{(L^{\infty}(\R))^{N^2}} \mathcal{K}_6\Delta t|\omega_{\epsilon}|_{BV(\R)}\\&\qquad \times\int_{0}^T|\omega_{\epsilon_0}(s-t)| \d s \d t \d y
				\\
				&\le{\int_{Q_T} \abs{f^k(\overline{U}^k(s,y)) } \abs{\nu^k}_{\lipR}\norma{\boldsymbol{\mu}}_{(L^{\infty}(\R))^{N^2}} \mathcal{K}_6\Delta t \frac{1}{\epsilon} \d s  \d y}\\
				&= \abs{\nu^k}_{\lipR}\norma{\boldsymbol{\mu}}_{(L^{\infty}(\R))^{N^2}} \mathcal{K}_6 \norma{f^k(\overline{U}^k)}_{L^1(Q_T)}\frac{\Delta t}{\epsilon}.\end{align*}
			Finally, combining the above estimates, we get:
			\begin{equation*}
				\lambda_3= \lambda_3' +\mathcal{O}\left(\frac{\D x}{\epsilon}\right).\end{equation*}
			Thus, so far we have proved
			\begin{equation}
				-\Lambda^k_{\Delta,\epsilon,\epsilon_0}(U^{k}_{\Delta},U^k)=\lambda_1+\lambda_2'+\lambda_3'+\mathcal{O}\left(\frac{\D x}{\epsilon}\right).
			\end{equation}
			Now, we estimate $\lambda_1+\lambda_2'+\lambda_3'$ as follows
			\begin{align*}
				\lambda_1+&\lambda_2'+\lambda_3'\nonumber\\
				&=\quad\int_{Q_T}\sum_{i\in \Z}\sum\limits_{n=0}^{N_T-1}  \left(\eta_i^{k,n+1}(\overline{U}^k(s,y))-\eta_i^{k,n}(\overline{U}^k(s,y))\right)\\&\qquad\qquad\qquad\qquad\times\left(\int_{C_i}\Phi(s,y,t^{n+1},x)  \d x\right)\d s  \d y \nonumber\\
				& \quad\nonumber -\int_{Q_T}\sum_{i\in \Z}\sum\limits_{n=0}^{N_T-1}  \int_{C_i^n}p_i^{k,n}(\overline{U}^k(s,y))\nu^k({\boldsymbol{c}}_{i+1/2}^{k,n})\Phi_x(s,y,t,x)  \d t\d x\d s  \d y\\
				&\quad\nonumber
				+\int_{Q_T}\sum_{i\in \Z}\sum\limits_{n=0}^{N_T-1}  \sgn\left(U_i^{k,n}-\frac{\overline{U}^k(s,y)}{ \sigma^k_i}\right) (\nu^k({\boldsymbol{c}}_{i+1/2}^{k,n})-\nu^k({\boldsymbol{c}}_{i-1/2}^{k,n}))\\
				&\nonumber\qquad\qquad\qquad  \qquad\times f^k(\overline{U}^k(s,y))\left(\int_{C^{n}} \Phi(s,y,t,x_{i+1/2})  \d t\right) \d s \d y .
			\end{align*}
			Using the discrete entropy inequality, cf.~Theorem~\ref{Existence} (2), 
			we get
			\begin{align*}
				&\lambda_1+\lambda_2'+\lambda_3'\\ &\hspace{-.1cm}\le-\lambda \int_{Q_T}\sum_{i\in \Z}\sum\limits_{n=0}^{N_T-1} 
				\big(G^{k,n}_{i+1/2}(U_i^{k,n} ,U_{i+1}^{k,n},\overline{U}^k(s,y))-G^{k,n}_{i-1/2}(U_{i-1}^{k,n} ,U_{i}^{k,n},\overline{U}^k(s,y))\big)\\
				&\qquad\qquad\qquad   \qquad\times\left(\int_{C_i}\Phi(s,y,t^{n+1},x)  \d x\right)\d s  \d y\\
				&\quad\hspace{-.1cm}-\lambda \int_{Q_T}\sum_{i\in \Z}\sum\limits_{n=0}^{N_T-1}\sgn\left(U_i^{k,n+1}-\frac{\overline{U}^k(s,y)}{ \sigma^k_i}\right)(\nu^k({\boldsymbol{c}}_{i+1/2}^{k,n})-\nu^k({\boldsymbol{c}}_{i-1/2}^{k,n}))\\
				&\qquad\qquad\qquad\qquad\times f^k(\overline{U}^k(s,y)) \left(\int_{C_i}\Phi(s,y,t^{n+1},x)  \d x\right)\d s  \d y\\
				& \quad\hspace{-.1cm}\nonumber -\int_{Q_T}\sum_{i\in \Z}\sum\limits_{n=0}^{N_T-1}  \int_{C_i^n}p_i^{k,n}(\overline{U}^k(s,y))\nu^k({\boldsymbol{c}}_{i+1/2}^{k,n})\Phi_x(s,y,t,x)  \d t\d x\d s  \d y\\
				&\quad\nonumber\hspace{-.1cm}
				+\int_{Q_T}\sum_{i\in \Z}\sum\limits_{n=0}^{N_T-1}  \sgn\left(U_i^{k,n}-\frac{\overline{U}^k(s,y)}{ \sigma^k_i}\right)(\nu^k({\boldsymbol{c}}_{i+1/2}^{k,n})-\nu^k({\boldsymbol{c}}_{i-1/2}^{k,n})) \\
				&\qquad\qquad\qquad\qquad\times f^k(\overline{U}^k(s,y))\left(\int_{C^{n}} \Phi(s,y,t,x_{i+1/2})  \d t\right) \d s \d y. \\
				&\hspace{-.1cm}=-\lambda \int_{Q_T}\sum_{i\in \Z}\sum\limits_{n=0}^{N_T-1}
				\big(G^{k,n}_{i+1/2}(U_i^{k,n} ,U_{i+1}^{k,n},\overline{U}^k(s,y))-G^{k,n}_{i-1/2}(U_{i-1}^{k,n} ,U_{i}^{k,n},\overline{U}^k(s,y))\big) \\
				&\qquad\qquad\qquad   \qquad\times\left(\int_{C_i}\Phi(s,y,t^{n+1},x)  \d x\right)\d s  \d y\\
				&\quad\hspace{-.1cm}-\lambda \int_{Q_T}\sum_{i\in \Z}\sum\limits_{n=0}^{N_T-1}\sgn\left(U_i^{k,n+1}-\frac{\overline{U}^k(s,y)}{ \sigma^k_i}\right)(\nu^k({\boldsymbol{c}}_{i+1/2}^{k,n})-\nu^k({\boldsymbol{c}}_{i-1/2}^{k,n})) 
				\\
				&\qquad\qquad\qquad\qquad  \times f^k(\overline{U}^k(s,y))\left(\int_{C_i}\Phi(s,y,t^{n+1},x)  \d x\right)\d s  \d y\\
				& \quad\hspace{-.1cm}\nonumber +\int_{Q_T}\sum_{i\in \Z}\sum\limits_{n=0}^{N_T-1}    \int_{C^{n}} \left(p_i^{k,n}(\overline{U}^k(s,y))\nu^k({\boldsymbol{c}}_{i+1/2}^{k,n})- p_{i-1}^{k,n}(\overline{U}^k(s,y))\nu^k({\boldsymbol{c}}_{i-1/2}^{k,n}) \right)\\
				&\qquad\qquad\qquad   \qquad\times \left(\int_{C^{n}} \Phi(s,y,t,x_{i+1/2})  \d t\right)\d s  \d y \\
				&\nonumber\quad\hspace{-.1cm}
				+\int_{Q_T}\sum_{i\in \Z}\sum\limits_{n=0}^{N_T-1}  \sgn\left(U_i^{k,n}-\frac{\overline{U}^k(s,y)}{ \sigma^k_i}\right)(\nu^k({\boldsymbol{c}}_{i+1/2}^{k,n})-\nu^k({\boldsymbol{c}}_{i-1/2}^{k,n})) \\
				&\qquad\qquad\qquad   \qquad\times \left(\int_{C^{n}} \Phi(s,y,t,x_{i+1/2})  \d t \right)\d s \d y \\
				&\hspace{-.1cm}:=A_2+A_3+\lambda_2'+\lambda_3',
			\end{align*}
			by using the fundamental theorem of calculus followed by summation by parts in $\lambda_2'$.
			
			\textbf{Claim 1.} $A_2+\lambda_2' =\mathcal{O}\left(\frac{\D x}{\epsilon}+\frac{\D t}{\epsilon_0}\right).$\\
			Adding and subtracting \begin{align*}&\lambda\int_{Q_T}\sum_{i\in \Z}\sum\limits_{n=0}^{N_T-1}(p_i^{k,n}(\overline{U}^k(s,y))\nu^k({\boldsymbol{c}}_{i+1/2}^{k,n})-p_{i-1}^{k,n}(\overline{U}^k(s,y))\nu^k({\boldsymbol{c}}_{i-1/2}^{k,n}))\\
				&\qquad\qquad\qquad \times \left(\int_{C_i}\Phi(s,y,t^{n+1},x)\d x\right)\d s \d y,
			\end{align*}  
			we have 
			\begin{align*}
				&A_2+\lambda'_2\\
				&\quad=-\lambda\int_{Q_T}\big(G^{k,n}_{i+1/2}(U_i^{k,n} ,U_{i+1}^{k,n},\overline{U}^k(s,y))-G^{k,n}_{i-1/2}(U_{i-1}^{k,n} ,U_{i}^{k,n},\overline{U}^k(s,y))\big)\\
				&\qquad\qquad\quad\times \left(\int_{C_i}\Phi(s,y,t^{n+1},x)\d x\right)\d s \d y\nonumber\\
				&\qquad+\int_{Q_T}\sum_{i\in \Z}\sum\limits_{n=0}^{N_T-1}(p_i^{k,n}(\overline{U}^k(s,y))\nu^k({\boldsymbol{c}}_{i+1/2}^{k,n})-p_{i-1}^{k,n}(\overline{U}^k(s,y))\nu^k({\boldsymbol{c}}_{i-1/2}^{k,n}))\nonumber\\
				&\qquad\qquad\qquad\qquad\qquad\times \left(\int_{C^{n}} \Phi(s,y,t,x_{i+1/2})\d t-\lambda\int_{C_i}\Phi(s,y,t^{n+1},x)\d x\right)\d s \d y\nonumber\\
				&\qquad+\lambda\int_{Q_T}\sum_{i\in \Z}\sum\limits_{n=0}^{N_T-1}(p_i^{k,n}(\overline{U}^k(s,y))\nu^k({\boldsymbol{c}}_{i+1/2}^{k,n})-p_{i-1}^{k,n}(\overline{U}^k(s,y))\nu^k({\boldsymbol{c}}_{i-1/2}^{k,n}))\\
				&\qquad\qquad\qquad\qquad\quad\qquad\times  \left(\int_{C_i}\Phi(s,y,t^{n+1},x)\d x\right)\d s \d y\nonumber\\
				&\quad=\lambda\int_{Q_T}\sum_{i\in \Z}\sum\limits_{n=0}^{N_T-1}(p_i^{k,n}(\overline{U}^k(s,y))\nu^k({\boldsymbol{c}}_{i+1/2}^{k,n})-G^{k,n}_{i+1/2}(U_i^{k,n} ,U_{i+1}^{k,n},\overline{U}^k(s,y)))\\
				&\qquad\qquad\qquad\quad\qquad\times \left(\int_{C_i}\Phi(s,y,t^{n+1},x)\d x\right)\d s \d y\nonumber\\
				&\qquad-\lambda\int_{Q_T}\sum_{i\in \Z}\sum\limits_{n=0}^{N_T-1}(p_{i-1}^{k,n}(\overline{U}^k(s,y))\nu^k({\boldsymbol{c}}_{i-1/2}^{k,n})-G^{k,n}_{i-1/2}(U_{i-1}^{k,n} ,U_{i}^{k,n},\overline{U}^k(s,y)))\\
				&\qquad\qquad\qquad\qquad\qquad\times \left(\int_{C_i}\Phi(s,y,t^{n+1},x)\d x\right)\d s \d y\nonumber\\
				&\qquad+\int_{Q_T}\sum_{i\in \Z}\sum\limits_{n=0}^{N_T-1}(p_i^{k,n}(\overline{U}^k(s,y))\nu^k({\boldsymbol{c}}_{i+1/2}^{k,n})-p_{i-1}^{k,n}(\overline{U}^k(s,y))\nu^k({\boldsymbol{c}}_{i-1/2}^{k,n}))\nonumber\\
				&\qquad\qquad\qquad\qquad\qquad\times \left(\int_{C^{n}} \Phi(s,y,t,x_{i+1/2})\d t-\lambda\int_{C_i}\Phi(s,y,t^{n+1},x)\d x\right)\d s \d y.
			\end{align*}
			Applying summation by parts we get
			\begin{align}
				A_2+\lambda'_2&=\lambda\int_{Q_T}\sum_{i\in \Z}\sum\limits_{n=0}^{N_T-1}(G^{k,n}_{i+1/2}(U_i^{k,n} ,U_{i+1}^{k,n},\overline{U}^k(s,y))-p_i^{k,n}(\overline{U}^k(s,y))\nu^k({\boldsymbol{c}}_{i+1/2}^{k,n}))\nonumber\\
				&\quad\qquad\times\left(\int_{C_{i+1}}\Phi(s,y,t^{n+1},x)\d x-\int_{C_i}\Phi(s,y,t^{n+1},x)\d x\right)\d s \d y\nonumber\\
				&\quad+\int_{Q_T}\sum_{i\in \Z}\sum\limits_{n=0}^{N_T-1}(p_i^{k,n}(\overline{U}^k(s,y))\nu^k({\boldsymbol{c}}_{i+1/2}^{k,n})-p_{i-1}^{k,n}(\overline{U}^k(s,y))\nu^k({\boldsymbol{c}}_{i-1/2}^{k,n}))\nonumber\\
				&\qquad\qquad\times\left(\int_{C^{n}} \Phi(s,y,t,x_{i+1/2})\d t-\lambda\int_{C_i}\Phi(s,y,t^{n+1},x)\d x\right)\d s \d y,\nonumber
			\end{align}
			where 
			\begin{align}
				\abs{G^{k,n}_{i+1/2}(U_i^{k,n} ,U_{i+1}^{k,n},\overline{U}^k(s,y))-\nu^k({\boldsymbol{c}}_{i+1/2}^{k,n}) p_i^{k,n}(\overline{U}^k(s,y))  }  &\leq \mathcal{M}_5   \abs{\Delta_{+}U_i^{k,n}}, \label{est:1} 
			\end{align}
			and 
			\begin{align}\label{est:2}
				\begin{split}
					\abs{\Delta_{-}(\nu^k({\boldsymbol{c}}_{i+1/2}^{k,n}) p_i^{k,n}(\overline{U}^k(s,y)))}  &\leq \mathcal{M}_6 \big[\abs{\Delta_{+}U_i^{k,n}}+\abs{\Delta_{-}(\nu^k({\boldsymbol{c}}_{i+1/2}^{k,n}))}  \big]. 
				\end{split}
			\end{align}
			Using \eqref{est:1}--\eqref{est:2}, we have,
			\begin{align*}
				A_2+\lambda_2'
				&\leq \mathcal{M}_5 \lambda\int_{Q_T}\sum_{i\in \Z}\sum\limits_{n=0}^{N_T-1} \abs{U_{i+1}^{k,n}-U_i^{k,n}}\\
				&\qquad\qquad\times \left(\int_{C_{i+1}}\Phi(s,y,t^{n+1},x)\d x-\int_{C_i}\Phi(s,y,t^{n+1},x)\d x\right)\d s \d y\nonumber\\
				&\quad+\mathcal{M}_6 \int_{Q_T}\sum_{i\in \Z}\sum\limits_{n=0}^{N_T-1}\big[\abs{\Delta_{+}U_i^{k,n}}+\abs{\Delta_{-}(\nu^k({\boldsymbol{c}}_{i+1/2}^{k,n}))}  \big]\nonumber\\
				&\qquad\qquad\times\left(\int_{C^{n}} \Phi(s,y,t,x_{i+1/2})\d t-\lambda\int_{C_i}\Phi(s,y,t^{n+1},x)\d x\right)\d s \d y. \nonumber
			\end{align*}
			Now as $\boldsymbol{\mu}$ is a smooth function with compact support, the claim follows because of \eqref{est:3}--\eqref{est:4} and the following estimates, 
			\begin{align}  &\nonumber\sum_{i}\Big|\nu^k({\boldsymbol{c}}_{i+1/2}^{k,n})-\nu^k({\boldsymbol{c}}_{i-1/2}^{k,n})\Big|\\
				&\quad\le\abs{\nu^k}_{\lipR}\sum_{j=1}^N\sum_{i,p} \Delta x\Big|(\boldsymbol{\mu}^{j,k}_{i+1/2-p}-\boldsymbol{\mu}^{j,k}_{i-1/2-p})U^{j,n}_{p+1/2}\Big|\nonumber\\
				&\quad\le \abs{\nu^k}_{\lipR}\sum_{j=1}^N\sum_{i,p} \Delta x\abs{U^{j,n}_{p+1/2}}\int_{C_{i-p}}\abs{{\boldsymbol{\mu}^{j,k}}^{'}(x)}\d x\nonumber\\
				&\quad=\abs{\nu^k}_{\lipR}\sum_{j=1}^N\sum_{p}\Delta x\abs{U^{j,n}_{p+1/2}}\int_{\R}\abs{{\boldsymbol{\mu}^{j,k}}^{'}(x)}\d x\nonumber\\
				&\quad\leq  \abs{\nu^k}_{\lipR}\norma{\boldsymbol{U}_{0}}_{(L^1(\R))^N}\abs{\boldsymbol{\mu}}_{(BV(\R))^{N^2}}:=\mathcal{M}_7\label{cc},
			\end{align}
			which is true because of \eqref{apx:L1}.\\
			\textbf{Claim 2.} $A_3+\lambda_3' =\mathcal{O}\left(\frac{\D x}{\epsilon}\right)$.\\
			Observe that, 
			\begin{align*}
				&A_3+\lambda_3'\\
				&\quad= -\lambda \int_{Q_T}\sum_{i\in \Z}\sum\limits_{n=0}^{N_T-1}\sgn\left(U_i^{k,n+1}-\frac{\overline{U}^k(s,y)}{ \sigma^k_i}\right)(\nu^k({\boldsymbol{c}}_{i+1/2}^{k,n})-\nu^k({\boldsymbol{c}}_{i-1/2}^{k,n}))\\
				&\qquad\qquad\qquad\qquad \qquad\times f^k(\overline{U}^k(s,y))\left(\int_{C_i}\Phi(s,y,t^{n+1},x)  \d x\right)\d s  \d y\\
				&\qquad+\int_{Q_T}\sum_{i\in \Z}\sum\limits_{n=0}^{N_T-1}  \sgn\left(U_i^{k,n}-\frac{\overline{U}^k(s,y)}{ \sigma^k_i}\right)(\nu^k({\boldsymbol{c}}_{i+1/2}^{k,n})-\nu^k({\boldsymbol{c}}_{i-1/2}^{k,n}))\\
				&\qquad\qquad\qquad\qquad \quad\times f^k(\overline{U}^k(s,y)) \left(\int_{C^{n}} \Phi(s,y,t,x_{i+1/2})  \d t\right) \d s \d y\\
				&\quad = -\lambda \int_{Q_T}\sum_{i\in \Z}\sum\limits_{n=0}^{N_T-1}\sgn\left(U_i^{k,n+1}-\frac{\overline{U}^k(s,y)}{ \sigma^k_i}\right)(\nu^k({\boldsymbol{c}}_{i+1/2}^{k,n})-\nu^k({\boldsymbol{c}}_{i-1/2}^{k,n}))\\
				&\qquad\qquad\qquad  \qquad \qquad\times f^k(\overline{U}^k(s,y)) \left(\int_{C_i}\Phi(s,y,t^{n+1},x)  \d x\right)\d s  \d y\\
				&\qquad+\lambda \int_{Q_T}\sum_{i\in \Z}\sum\limits_{n=0}^{N_T-1}\sgn\left(U_i^{k,n}-\frac{\overline{U}^k(s,y)}{ \sigma^k_i}\right)(\nu^k({\boldsymbol{c}}_{i+1/2}^{k,n})-\nu^k({\boldsymbol{c}}_{i-1/2}^{k,n})) 
				\\
				&\qquad\qquad\qquad   \qquad \qquad\times f^k(\overline{U}^k(s,y))
				\int_{C_i}\Phi(s,y,t^{n+1},x)  \d x\d s  \d y\\
				&\qquad-\lambda \int_{Q_T}\sum_{i\in \Z}\sum\limits_{n=0}^{N_T-1}\sgn\left(U_i^{k,n}-\frac{\overline{U}^k(s,y)}{ \sigma^k_i}\right)(\nu^k({\boldsymbol{c}}_{i+1/2}^{k,n})-\nu^k({\boldsymbol{c}}_{i-1/2}^{k,n}))\\
				&\qquad\qquad\qquad   \qquad \qquad\times f^k(\overline{U}^k(s,y)) \int_{C_i}\Phi(s,y,t^{n+1},x)  \d x\d s  \d y\\
				&\qquad \quad+\int_{Q_T}\sum_{i\in \Z}\sum\limits_{n=0}^{N_T-1}  \sgn\left(U_i^{k,n}-\frac{\overline{U}^k(s,y)}{ \sigma^k_i}\right)(\nu^k({\boldsymbol{c}}_{i+1/2}^{k,n})-\nu^k({\boldsymbol{c}}_{i-1/2}^{k,n})) \\
				&\qquad\qquad\qquad   \qquad \qquad\times f^k(\overline{U}^k(s,y))\int_{C^{n}} \Phi(s,y,t,x_{i+1/2})  \d s \d t \d y\\
				& \quad=
				\int_{Q_T} \sum_{i\in \Z}\sum\limits_{n=0}^{N_T-1}\sgn\left(U_i^{k,n}-\frac{\overline{U}^k(s,y)}{ \sigma^k_i}\right)(\nu^k({\boldsymbol{c}}_{i+1/2}^{k,n})-\nu^k({\boldsymbol{c}}_{i-1/2}^{k,n}))f^k(\overline{U}^k(s,y)) \\
				&\qquad\qquad\qquad\qquad\times \left(\int_{C^{n}} \Phi(s,y,t,x_{i+1/2})\d t    -\lambda\int_{C_i}\Phi(s,y,t^{n+1},x) \d x \right) \d s \d y\\
				&\qquad-\lambda \int_{Q_T} \sum_{i\in \Z}\sum\limits_{n=0}^{N_T-1}(\nu^k({\boldsymbol{c}}_{i+1/2}^{k,n})-\nu^k({\boldsymbol{c}}_{i-1/2}^{k,n}))f^k(\overline{U}^k(s,y))\\
				&\qquad\qquad\qquad   \qquad \qquad\times \left(\sgn\left(U_i^{k,n+1}-\frac{\overline{U}^k(s,y)}{ \sigma^k_i}\right)-\sgn\left(U_i^{k,n}-\frac{\overline{U}^k(s,y)}{ \sigma^k_i}\right)\right)\\
				&\qquad\qquad\qquad   \qquad \qquad\times 
				\left(\int_{C_i}\Phi(s,y,t^{n+1},x)  \d x \right) \d s \d y\\
				&\quad:=\tilde{\mathcal{E}}_1+\tilde{\mathcal{E}}_2.
			\end{align*}
			Now, using \eqref{cc} and \eqref{est:3},
			\begin{align*}
				&\tilde{\mathcal{E}}_1
				\le T\norma{f^k(\overline{U}^k)}_{L^1(Q_T)}\mathcal{M}_7 \left(\frac{\Delta x}{\epsilon}+\frac{\Delta t}{\epsilon_0}\right).
			\end{align*}
			Furthermore,
			\begin{align*}
				\tilde{\mathcal{E}}_2
				&=-\lambda \int_{Q_T}\sum_{i}\sum_{n=1}^{N_T}(\nu^k({\boldsymbol{c}}_{i+1/2}^{k,n-1})-\nu^k({\boldsymbol{c}}_{i-1/2}^{k,n-1}))\sgn\left(U_i^{k,n}-\frac{\overline{U}^k(s,y)}{ \sigma^k_i}\right) \\
				&\qquad\qquad\qquad \qquad\times  f^k(\overline{U}^k(s,y))
				\left(\int_{C_i}\Phi(s,y,t^{n},x)  \d x\right)\d s  \d y\\
				&\quad+\lambda \int_{Q_T}\sum_{i}\sum_{n=1}^{N_T}\sgn\left(U_i^{k,n}-\frac{\overline{U}^k(s,y)}{ \sigma^k_i}\right)(\nu^k({\boldsymbol{c}}_{i+1/2}^{k,n})-\nu^k({\boldsymbol{c}}_{i-1/2}^{k,n})) \\
				&\qquad\qquad\qquad \qquad\times f^k(\overline{U}^k(s,y)) 
				\int_{C_i}\Phi(s,y,t^{n+1},x)  \d x\d s  \d y\\
				&\quad +\lambda \int_{Q_T}\sum_{i}(\nu^k({\boldsymbol{c}}_{i+1/2}^{k,0})-\nu^k({\boldsymbol{c}}_{i-1/2}^{k,0}))\sgn\left(U_i^{k,0}-\frac{\overline{U}^k(s,y)}{ \sigma^k_i}\right)
				\\
				&\qquad\qquad \qquad\times  f^k(\overline{U}^k(s,y)) 
				\left(\int_{C_i}\Phi(s,y,t^1,x)  \d x\right)\d s  \d y\\
				&\quad-\lambda \int_{Q_T}\sum_{i}(\nu^k({\boldsymbol{c}}_{i+1/2}^{k,N_T})-\nu^k({\boldsymbol{c}}_{i-1/2}^{k,N_T}))\sgn\left(U_i^{k,N_T}-\frac{\overline{U}^k(s,y)}{ \sigma^k_i}\right)  \\
				&\qquad\qquad\qquad\times   f^k(\overline{U}^k(s,y))
				\left(\int_{C_i}\Phi(s,y,t^{N_T+1},x)  \d x\right)\d s  \d y.\end{align*}
			Adding and subtracting 
			\begin{align*}
				&-\lambda \int_{Q_T}\sum_{i}\sum_{n=1}^{N_T}(\nu^k({\boldsymbol{c}}_{i+1/2}^{k,n-1})-\nu^k({\boldsymbol{c}}_{i-1/2}^{k,n-1}))\sgn\left(U_i^{k,n}-\frac{\overline{U}^k(s,y)}{ \sigma^k_i}\right) \\
				&\qquad\qquad\qquad\quad\times 
				f^k(\overline{U}^k(s,y))\left(\int_{C_i}\Phi(s,y,t^{n+1},x)  \d x\right)\d s  \d y,
			\end{align*} 
			we have,
			\begin{align*}
				\tilde{\mathcal{E}}_2&=-\lambda \int_{Q_T} \sum_{i}\sum_{n=1}^{N_T}(\nu^k({\boldsymbol{c}}_{i+1/2}^{k,n-1})-\nu^k({\boldsymbol{c}}_{i-1/2}^{k,n-1}))\sgn\left(U_i^{k,n}-\frac{\overline{U}^k(s,y)}{ \sigma^k_i}\right) \\
				&\qquad\qquad\qquad\times f^k(\overline{U}^k(s,y))
				\int_{C_i}\left((\Phi(s,y,t^{n},x)-\Phi(s,y,t^{n+1},x)) \d x\right) \d s \d y\\
				& \quad+\lambda \int_{Q_T} \sum_{i}\sum_{n=1}^{N_T}\sgn\left(U_i^{k,n}-\frac{\overline{U}^k(s,y)}{ \sigma^k_i}\right)(\nu^k({\boldsymbol{c}}_{i+1/2}^{k,n})-\nu^k({\boldsymbol{c}}_{i-1/2}^{k,n}))\\
				&\qquad\qquad\qquad\times f^k(\overline{U}^k(s,y))\left(\int_{C_i}\Phi(s,y,t^{n+1},x) \d x\right)\d s  \d y\\
				& \quad-\lambda \int_{Q_T} \sum_{i}\sum_{n=1}^{N_T}(\nu^k({\boldsymbol{c}}_{i+1/2}^{k,n-1})-\nu^k({\boldsymbol{c}}_{i-1/2}^{k,n-1}))\sgn\left(U_i^{k,n}-\frac{\overline{U}^k(s,y)}{ \sigma^k_i}\right)\\
				&\qquad\qquad\qquad\times f^k(\overline{U}^k(s,y))\left(\int_{C_i}\Phi(s,y,t^{n+1},x) \d x\right)\d s  \d y\\
				&\quad+\lambda \int_{Q_T}  \sum_{i}(\nu^k({\boldsymbol{c}}_{i+1/2}^{k,0})-\nu^k({\boldsymbol{c}}_{i-1/2}^{k,0}))\left(\sgn(U_i^{k,0}-\frac{\overline{U}^k(s,y)}{ \sigma^k_i}\right) 
				\\
				&\qquad\qquad\quad\times f^k(\overline{U}^k(s,y)) \left(\int_{C_i}\Phi(s,y,t^1,x)  \d x \right)\d s \d y\\
				&\quad-\lambda \int_{Q_T} \sum_{i}(\nu^k({\boldsymbol{c}}_{i+1/2}^{k,N_T})-\nu^k({\boldsymbol{c}}_{i-1/2}^{k,N_T}))\sgn\left(U_i^{k,N_T}-\frac{\overline{U}^k(s,y)}{ \sigma^k_i}\right) 
				\\
				&\qquad\qquad\quad \times f^k(\overline{U}^k(s,y))\left(\int_{C_i}\Phi(s,y,t^{N_T+1},x) \d x \right) \d s \d y\\
				&\quad:=\tilde{\mathcal{E}}_{21}+\tilde{\mathcal{E}}_{22}+\tilde{\mathcal{E}}_{23}+\tilde{\mathcal{E}}_{24}+\tilde{\mathcal{E}}_{25}.
			\end{align*}
			Now, let us estimate $\tilde{\mathcal{E}}_{21}$:
			\begin{align*}
				\tilde{\mathcal{E}}_{21}
				& = -\lambda \int_{Q_T} \Bigg(\sum_{i}\sum_{n=1}^{N_T}(\nu^k({\boldsymbol{c}}_{i+1/2}^{k,n-1})-\nu^k({\boldsymbol{c}}_{i-1/2}^{k,n-1}))\sgn\left(U_i^{k,n}-\frac{\overline{U}^k(s,y)}{ \sigma^k_i}\right) \\
				&\qquad\qquad\times f^k(\overline{U}^k(s,y))
				\left(\int_{C_i}(\Phi(s,y,t^{n},x)-\Phi(s,y,t^{n+1},x)) \d x \right)\Bigg) \d s  \d y\\
				&\leq \D x \lambda \norma{f^k(\overline{U}^k)}_{L^{\infty} (Q_T)} \int_{0}^T \sum_{i}\sum_{n=1}^{N_T} 
				\abs{\nu^k({\boldsymbol{c}}_{i+1/2}^{k,n-1})-\nu^k({\boldsymbol{c}}_{i-1/2}^{k,n-1})}\\
				&\qquad\qquad\qquad \qquad\qquad\qquad\qquad\qquad\times
				\abs{\omega_{\epsilon_0}(s-t^n)-\omega_{\epsilon_0}(s-t^n-\Delta t)} \d s \\
				&\le \D t\mathcal{M}_7  \norma{f^k(\overline{U}^k)}_{L^{\infty}(Q_T)} \sum_{n=1}^{N_T} \int_{0}^T 
				\abs{\omega_{\epsilon_0}(s-t^n)-\omega_{\epsilon_0}(s-t^n-\Delta t)} \d s \\
				&\le \mathcal{M}_7\Delta t\norma{f^k(\overline{U}^k)}_{L^{\infty}(Q_T)} \sum_{n=1}^{N_T}  \abs{\omega_{\epsilon_0}}_{BV(\R)}\Delta t
				:=\mathcal{M}_8\frac{\Delta t}{\epsilon_0}.
			\end{align*}
			Next, we consider $\tilde{\mathcal{E}}_{22}+\tilde{\mathcal{E}}_{23}$:
			\begin{align*}
				\tilde{\mathcal{E}}_{22}+\tilde{\mathcal{E}}_{23}
				&=\lambda \int_{Q_T} \sum_{i}\sum_{n=1}^{N_T}\sgn\left(U_i^{k,n}-\frac{\overline{U}^k(s,y)}{ \sigma^k_i}\right)(\nu^k({\boldsymbol{c}}_{i+1/2}^{k,n})-\nu^k({\boldsymbol{c}}_{i-1/2}^{k,n}))\\
				&\qquad\qquad\qquad\quad\times f^k(\overline{U}^k(s,y))\left(\int_{C_i}\Phi(s,y,t^{n+1},x) \d x\right)\d s  \d y\\
				&\quad-\lambda \int_{Q_T} \sum_{i}\sum_{n=1}^{N_T}(\nu^k({\boldsymbol{c}}_{i+1/2}^{k,n-1})-\nu^k({\boldsymbol{c}}_{i-1/2}^{k,n-1}))\sgn\left(U_i^{k,n}-\frac{\overline{U}^k(s,y)}{ \sigma^k_i}\right)\\
				&\qquad\qquad\qquad\qquad\times f^k(\overline{U}^k(s,y))\left(\int_{C_i}\Phi(s,y,t^{n+1},x) \d x\right)\d s  \d y\\
				&\le\lambda \abs{\nu^k}_{\lipR}\norma{f^k(\overline{U}^k)}_{L^{\infty}(Q_T)} \\
				&\qquad\times \int_{Q_T} \Bigg(\sum_{i}\sum_{n=1}^{N_T}\Big|{\boldsymbol{c}}_{i+1/2}^{k,n}-{\boldsymbol{c}}_{i-1/2}^{k,n}-{\boldsymbol{c}}_{i+1/2}^{k,n-1}+{\boldsymbol{c}}_{i-1/2}^{k,n-1}\Big| \\
				&\qquad\qquad\qquad\qquad\quad\times \int_{C_i}|\Phi(s,y,t^{n+1},x)| \d x \Bigg)\d s \d y\\
				&\quad=\Delta t\abs{\nu^k}_{\lipR}\norma{f^k(\overline{U}^k)}_{L^{\infty}(Q_T)}  \\
				&\qquad\quad\times 
				\sum_{i}\sum_{n=1}^{N_T}\Big|{\boldsymbol{c}}_{i+1/2}^{k,n}-{\boldsymbol{c}}_{i-1/2}^{k,n}-{\boldsymbol{c}}_{i+1/2}^{k,n-1}+{\boldsymbol{c}}_{i-1/2}^{k,n-1}\Big|.
			\end{align*}
			Further,
			\begin{align*}  
				&\sum_{i}\Big|(\boldsymbol{c}_{i+1/2}^{k,n}-\boldsymbol{c}_{i-1/2}^{k,n})-(\boldsymbol{c}_{i+1/2}^{k,n-1}-{\boldsymbol{c}}_{i-1/2}^{k,n-1})\Big|\\
				&\qquad\le \Delta x \sum_{i}\Big|\sum_{j=1}^N\sum_{p} (\boldsymbol{\mu}^{j,k}_{i+1/2-p}-\boldsymbol{\mu}^{j,k}_{i-1/2-p})U^{j,n}_{p+1/2}\\
				&\qquad\qquad\qquad\qquad\qquad\qquad-\sum_{j=1}^N\sum_{p} (\boldsymbol{\mu}^{j,k}_{i+1/2-p}-\boldsymbol{\mu}^{j,k}_{i-1/2-p})U^{j,n-1}_{p+1/2}\Big|\\
				&\qquad\le\Delta x{\sum_{j=1}^N\sum_{i,p} \Big|U^{j,n}_{p+1/2}-U^{j,n-1}_{p+1/2}\Big|\Big|\boldsymbol{\mu}^{j,k}_{i+1/2-p}-\boldsymbol{\mu}^{j,k}_{i-1/2-p}\Big|}\\
				&\qquad\le \Delta x|\boldsymbol{\mu}|_{(BV(\R))^{N^2}}\sum_{j=1}^N\sum_{p}\Big|U_{p+1/2}^{j,n}-U_{p+1/2}^{j,n-1}\Big|\\
				&\qquad\le N|\boldsymbol{\mu}|_{(BV(\R))^{N^2}}\mathcal{K}_6\Delta t,
			\end{align*}
			see \eqref{apx:time}.
			Therefore,
			\begin{equation*}
				\tilde{\mathcal{E}}_{22}
				\le \mathcal{M}_8\Delta t. \end{equation*}
			The estimates on the remaining terms $\tilde{\mathcal{E}}_{23}$ and $\tilde{\mathcal{E}}_{24}$, follow easily from \eqref{cc}. Finally, the lemma follows combing all the above assertions.
		\end{proof}
		\begin{theorem}\label{CR}
			Let $\boldsymbol{U}$ be the entropy solution and $\boldsymbol{U}_{\Delta}$ be the numerical approximation obtained via then we have the following convergence rate estimates hold:
			\begin{align}\label{rate1}
				\norma{\boldsymbol{U}-\boldsymbol{U}_{\Delta}}_{(L^1(\R))^N} = \mathcal{O}(\D t)^{1/3}.
			\end{align}
			In particular, if $\TV( s)=0$ then we have,   
			\begin{align}\label{rate2}
				\norma{\boldsymbol{U}-\boldsymbol{U}_{\Delta}}_{(L^1(\R))^N} = \mathcal{O}(\D t)^{1/2}.
			\end{align}
		\end{theorem}
		\begin{proof}
			Owing to the CFL condition  (cf.~\eqref{CFL_LF} and \eqref{CFL_God} for the Lax--Friedrichs and Godunov type scheme, respectively), $\D x =\mathcal{O} (\D t).$
			Furthermore, the initial approximation \eqref{initial} implies $\norma{\boldsymbol{U}_0^{\D}-\boldsymbol{U}_0}_{(L^1(\R))^N}=\mathcal{O}(\D t)$.
			Observe that for $\TV( s) \neq 0,$ the Kuznetsov-type estimate, cf.~\eqref{est:kuz}, contains $\epsilon/\epsilon_0.$ Hence we choose $\epsilon= (\D t)^{2/3}$ and $\epsilon_0= (\D t)^{1/3}$ and finally substitute the assertions of Lemma \ref{lemma:est} and Theorem~\ref{Existence} (2)
			in \eqref{est:kuz} to get the convergence rate estimate \eqref{rate1}.
			On the other hand, if $\TV( s)=0,$ then the term $\frac{\epsilon}{\epsilon_0}$ will not appear in the Kuznetov-type estimate, cf.~\eqref{est:kuz}, allowing us to choose  $\epsilon=\epsilon_0=\sqrt{\Delta t}$ to get the better convergence rate estimate \eqref{rate2}.
		\end{proof}
		\begin{remark}\label{multid}\normalfont
			Kuznestov-type lemma \ref{lemma:kuz} remains valid in several space dimensions as well,  however for $\sigma$ discontinuous, a uniform $\bv$ bound on the finite volume approximation $\boldsymbol{U}_{\Delta}$ is not available. On the other hand, for $\sigma\equiv 1,$ where the system in 2 space dimensions, reads as
			, where the $k^{th}$ equation read as,
			\begin{align}
				\label{eq:u1A}
				\partial_t U^{k} +\partial_x \Big(f^k(U^k) \nu^k((\boldsymbol{\mu} * \boldsymbol{U})^k)\Big)+\partial_y \Big(g^k(U^k) \overline{\nu}^k((\boldsymbol{\overline{\mu}} * \boldsymbol{U})^k)\Big)&=0,\,\,k=1,2,\ldots,N,
			\end{align} such bounds on the numerical approximation can be derived, see \cite{ACG2015}, and consequently Lemma~\ref{lemma:est} and Theorem~\ref{CR} can be reproduced using dimension splitting arguments (see  \cite[Sec.~4.3]{HR2015}, \cite[Sec.~5]{HKLR2010} for any space dimension. We illustrate the experimental convergence rates numerically for two space dimensions in the next section.
		\end{remark}
		%-----  section
		\section{Numerical Experiments}
		\label{num}
		We now present some numerical experiments to illustrate the theory presented in the previous sections. We show the results for the Lax--Friedrichs scheme. The results obtained by  the Godunov scheme  are similar, and are not shown here. Throughout the section, $\theta=\theta_x=\theta_y$ is chosen to be $1/3$, and $\lambda$ and $\lambda_x=\lambda_y$ are chosen so as to satisfy the CFL condition \eqref{CFL_LF}.
		\subsection{One dimension}
		We consider the \textit{discontinuous} generalization of the nonlocal Keyfitz--Kranzer system introduced in \cite{ACG2015} in one space dimension, which reads as  
		\begin{equation}
			\label{eq:kk}
			\left\{
			\begin{array}{l}
				\partial_t U^1 + \partial_x \left(\sigma(x) U^1 \,  \nu (\mu*U^1,\mu*U^2)\right) = 0,
				\\
				\partial_t U^2 + \partial_x \left(\sigma(x) U^2 \,  \nu (\mu*U^1,\mu*U^2)\right) = 0.
			\end{array}
			\right.
		\end{equation}
		We choose the kernel
		\begin{align*}
			\mu(x)&=L(-x(0.125+x))^{5/2}\mathbbm{1}_{(-0.125,0)}(x),
		\end{align*}
		where $L$ is such that $\int_{\R}\mu(x)\d{x}=1,$ the velocity
		$\nu(a,b)=(1-a^2-b^2)^3,$ and a $\bv$ function
		\begin{align*}\label{s}
			\quad \sigma(x)= 
			\begin{cases}
				\frac{1}{3}a_1, \quad & x \le a_1,\\
				\frac{1}{3}\sum\limits_{n=2}^\infty a_n \mathbbm{1}_{[a_n,a_{n+1}]}(x), \quad & x \in (a_1,3],\\
				1, \quad & x> 3,\\
			\end{cases} a_n=3(1-0.8^n),
		\end{align*}
		which admits infinitely many discontinuities at the points $\{a_n\}_{n\in \N}.$
		The system~\eqref{eq:kk} fits into the setup of this article (cf.~\eqref{eq:umulA})  with $N=2, \boldsymbol{\mu} =
		\left[\begin{array}{ccc}
			\mu & \mu
			\\
			\mu & \mu
		\end{array}\right],\nu^k=\nu,f^k(u)=u$ and $\sigma^k=\sigma.$\begin{figure}[h!]
			\centering
			\begin{subfigure}{.45\textwidth}
				\includegraphics[width=\textwidth,keepaspectratio]{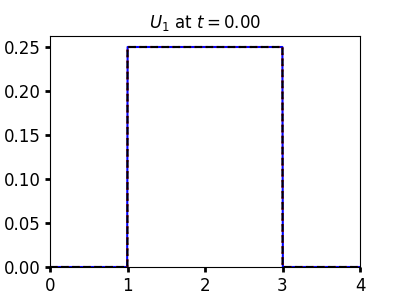}
			\end{subfigure}
			\hfill
			\begin{subfigure}{.45\textwidth}
				\includegraphics[width=\textwidth,keepaspectratio]{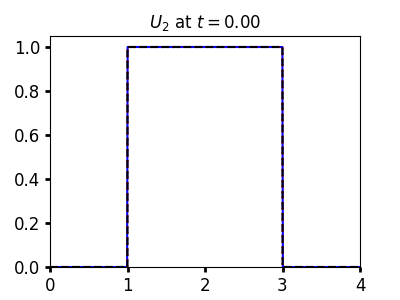}
			\end{subfigure}\hspace{1cm}
			\begin{subfigure}{.45\textwidth}
				\includegraphics[width=\textwidth,keepaspectratio]{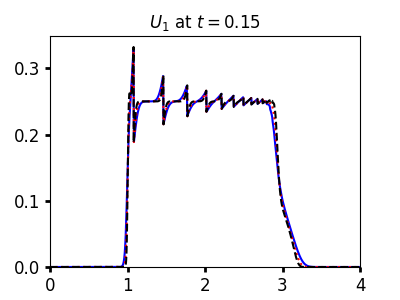}
			\end{subfigure}
			\hfill
			\begin{subfigure}{.45\textwidth}
				\includegraphics[width=\textwidth,keepaspectratio]{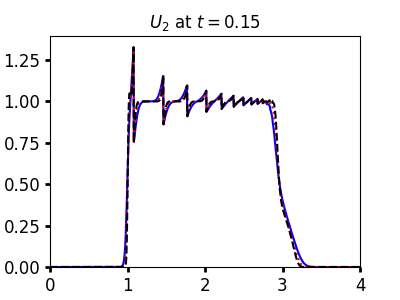}
			\end{subfigure}
			\begin{subfigure}{.45\textwidth}
				\includegraphics[width=\textwidth,keepaspectratio]{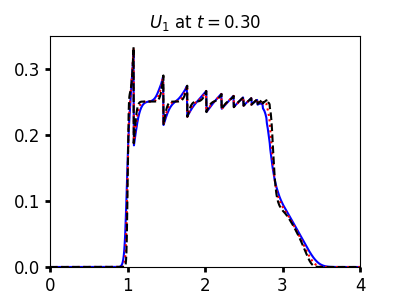}
			\end{subfigure}
			\hfill
			\begin{subfigure}{.45\textwidth}
				\includegraphics[width=\textwidth,keepaspectratio]{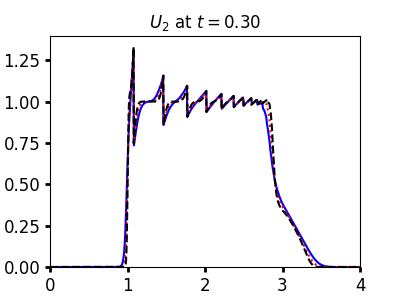}
			\end{subfigure}
			\caption{Solution to ~\eqref{eq:kk} and \eqref{eq:ex1} on the domain $[0, \, 4]$ at times $t =
				0.00,\; 0.15,\;0.3$ with decreasing mesh size  $\Delta x=0.00625 $({\color{blue}\full}), $\Delta x=0.00625/2 $({\color{red}\dotted}) and $\Delta x=0.00625/4 $({\color{black}\dashed})}
			\label{fig:ex1}
		\end{figure} 
		The local counterpart of \eqref{eq:kk}, viz.
		\begin{equation}
			\label{eq:kk1}
			\partial_t U^k + \partial_x \left(\sigma(x) U^k \,  \nu (U^1,U^2) \right)= 0, \quad k\in\{1,2\},
		\end{equation} 
		has not been studied yet in the literature and is of independent interest. However, the case $\sigma\equiv 1,$ is well understood, see, for example, \cite{KR2013}. 
		\begin{figure}[h!]
			\centering \noindent\begin{minipage}{0.4\textwidth}
				\centering
				\begin{tabular}{|c|c|c|c|c|c|c|c|c|c|}\hline
					\multicolumn{1}{|c|}{ $ \frac{\D x}{0.00625}$}&\multicolumn{1}{|c|}{$e_{\Delta x}(T)$}\vline & \multicolumn{1}{|c|}{$\alpha$}\vline\\
					\hline
					$1$&$5.85e-4$&$1.32$\tabularnewline
					\hline
					$1/2$&$2.34e-4$&$1.31$\tabularnewline
					\hline
					$1/4$&$9.48e-5$&$1.18$\tabularnewline
					\hline
					$1/8$&$4.16e-5$&\tabularnewline
					\hline
				\end{tabular}
			\end{minipage}
			\noindent\begin{minipage}{0.5\textwidth}
				\includegraphics[width=\textwidth, trim = 40 25 20 5]{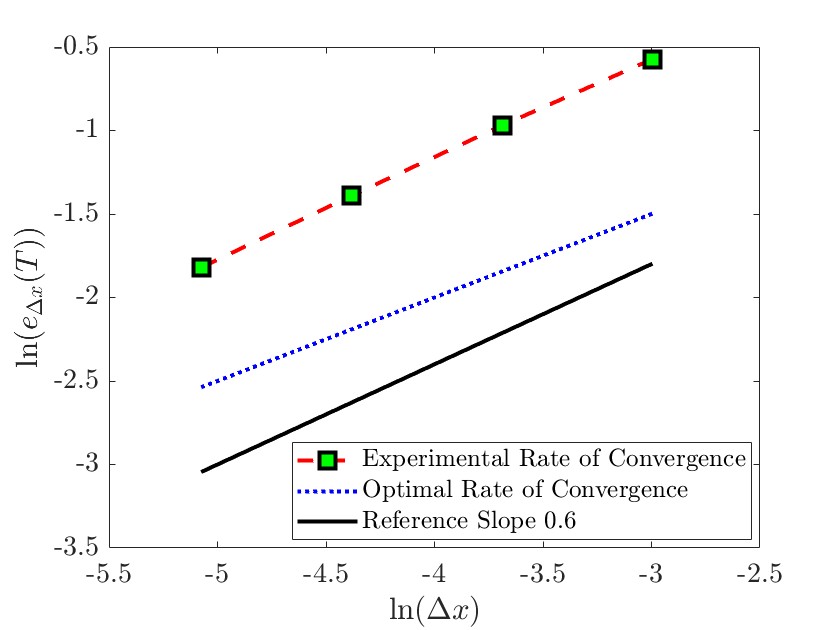}
			\end{minipage}  \caption{Convergence rate $\alpha$ for the numerical scheme~\eqref{scheme} on the domain $[0,\,4]$ at time $T=0.5$ for the approximate solutions to the problem~\eqref{eq:kk}, \eqref{eq:ex1}}\label{fig:my_label21}
		\end{figure}
		We present the numerical integration of~\eqref{eq:kk} on the domain
		$[0, \, 4]$ and the time
		interval $[0, \, 0.3]$ with
		\begin{align}
			\label{eq:ex1} U^1_0(x)=0.25\mathbbm{1}_{(1,3)}(x), \quad &
			U^2_0(x)=\mathbbm{1}_{(1,3)}(x).
		\end{align}
		Figure \ref{fig:ex1} displays the numerical approximations of \eqref{eq:kk} and \eqref{eq:ex1} generated by the numerical  
		scheme \eqref{scheme}, 
		with
		decreasing grid size $\Delta x$, starting with $\Delta x =0.00625$. It can be seen that the numerical scheme is able to capture both shocks and rarefactions well. Let $e_{\Delta x}(T)=\norma{\boldsymbol{U}_{\Delta x}(T,\dott)-\boldsymbol{U}_{\Delta x/2}(T,\dott)}_{(L^1(\R))^N}.$ The convergence rate reads $\alpha=\log_{2}(e_{\Delta x}(T)/{e_{\Delta x/2}(T)}).$
		Errors and the convergence rates at time $T=0.3$ obtained in the above experiment are recorded in
		Figure \ref{fig:my_label21}. The convergence rates thus obtained obey the claims of  Theorem \ref{CR}.
		The results recorded in 
		Figure \ref{fig:my_label21} show that the observed convergence rates exceed $1.$
		\subsection{Two dimensions}
		As pointed out earlier, the analysis of the paper fits in multiple space dimensions with $\boldsymbol{\sigma}\equiv 1$. To illustrate our results in two dimensions, we employ the two dimensional \textit{nonlocal} generalization of Keyfitz--Kranzer system,  introduced
		in~\cite{ACG2015}, given by
		\begin{equation}
			\label{eq:ex2}
			\left\{
			\begin{array}{l}
				\partial_t U^1
				+
				\partial_x \left(U^1 \,  \phi_1 (U^1 * \mu,U^2 * \mu)\right)
				+
				\partial_y \left(U^1 \,  \phi_2 (U^1 * \mu,U^2 * \mu)\right) = 0,
				\\
				\partial_t U^2
				+
				\partial_x \left(U^2 \,  \phi_1 (U^1 * \mu,U^2 * \mu)\right)
				+
				\partial_y \left(U^2 \,  \phi_2 (U^1 * \mu,U^2 * \mu)\right) = 0,
				\\
			\end{array}
			\right.
		\end{equation}
		which is a system of two nonlocal conservation laws in two space
		dimensions. Here, we choose
		$ \phi_1 (a,b) = \sin(a^2 + b^2),
		\phi_2 (a,b) = \cos(a^2+b^2)$
		and
		\begin{equation*}
			\label{eq:ex2:mu}
			\mu (x,y) =
			L\left(\eta^2 - (x^2 + y^2)\right)^3
			\mathbbm{1}_{\strut \{(x,y) \colon x^2 + y^2 \leq \eta^2\}} (x,y)\,,
		\end{equation*}
		where $L$ is chosen so that $\displaystyle\iint_{\R^2} \mu (x,y) \d{x} \d{y}=1$. The system fits into the setup of this article (cf.~\eqref{eq:umulA})  with $N=2, \boldsymbol{\mu} =
		\boldsymbol{\overline{\mu}}=\left[\begin{array}{ccc}
			\mu & \mu
			\\
			\mu & \mu
		\end{array}\right],\nu^k=\phi_1,\overline{\nu}^k=\phi_2, f^k(u)=g^k(u)=u$ and $\sigma^k=1.$ We refer to \cite{ACG2015} for the numerical integration of~\eqref{eq:ex2} on the domain
		the domain $[-1.1, \, 1.1]^2$ and to the time
		interval $[0, \, 0.1]$ with\\
		\noindent\begin{minipage}{0.45\linewidth}
			\includegraphics[width=\textwidth, trim = 40 25 20 5]{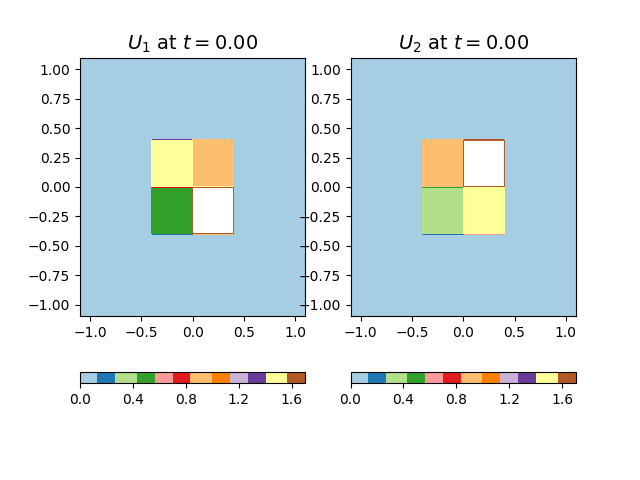}%
		\end{minipage}%
		\begin{minipage}{0.55\linewidth}
			\begin{equation}
				\label{eq:ex2:id}
				\!\!\!\!
				U_o (x,y)
				=
				\left\{
				\begin{array}{@{}ll@{}}
					(1, \, \sqrt{3}) &
					(x,y) \in \left] 0, \, 0.4 \right] \times \left]0, \, 0.4\right]
					\\
					(\sqrt{2}, \, 1) &
					(x,y) \in \left]-0.4, \, 0\right] \times \left]0, \, 0.4\right]
					\\
					(\frac{1}{2}, \, \frac{1}{3}) &
					(x,y) \in \left]-0.4, \, 0\right] \times \left]-0.4, \, 0\right]
					\\
					(\sqrt{3},\, \sqrt{2}) &
					(x,y) \in \left]0, \, 0.4\right] \times \left]-0.4, \, 0\right]
					\\
					(0, \, 0) &
					\mbox{elsewhere}
				\end{array}
				\right.
				\!\!\!\!\!\!\!\!\!\!\!\!
			\end{equation}
		\end{minipage}\\ \begin{figure}[h!]
			\centering
			\noindent\begin{minipage}{0.4\textwidth}
				\centering
				\begin{tabular}{|c|c|c|c|c|c|c|c|c|c|}\hline
					\multicolumn{1}{|c|}{ $\frac{\D x}{0.05}$}&\multicolumn{1}{|c|}{$e_{\Delta x}(T)$}\vline & \multicolumn{1}{|c|}{$\alpha$}\vline\\
					\hline
					$1$&$5.64e-1$&$0.57$\tabularnewline
					\hline
					$1/2$&$3.79e-1$&$0.60$\tabularnewline
					\hline
					$1/4$&$2.49e-1$&$0.62$\tabularnewline
					\hline
					$1/8$&$1.62e-1$&\tabularnewline
					\hline
				\end{tabular}
			\end{minipage}
			\noindent\begin{minipage}{0.5\textwidth}
				\includegraphics[width=\textwidth, trim = 40 25 20 5]{Pictures/error_2d_AHv2.jpg}
			\end{minipage}
			\caption{Convergence rate $\alpha$ for the numerical scheme~\eqref{scheme} on the domain $[-1.1, \, 1.1]^2$ at time $T=0.1$ for the approximate solutions
				to the problem~\eqref{eq:ex2}--\eqref{eq:ex2:id}.}\label{fig:my_label211}
		\end{figure}and present the convergence rate analysis in this paper, with $\eta=0.4.$ Errors and the convergence rates obtained in the experiment are recorded in Figure \ref{fig:my_label211}. The convergence rates thus obtained obey the claims of  Theorem \ref{CR}. The results recorded in Figures \ref{fig:my_label21}--\ref{fig:my_label211} show
		that the observed convergence rates exceed $0.5$.
		\bibliographystyle{siam}
		\bibliography{AHV_2}

\def\cprime{$'$}
\begin{thebibliography}{10}

\bibitem{AJV2004}
{\sc Adimurthi, J.~Jaffr{\'e}, and G.~D. Veerappa~Gowda}, {\em {G}odunov-type
  methods for conservation laws with a flux function discontinuous in space},
  SIAM J. Numer. Anal., 42 (2004), pp.~179--208.

\bibitem{AMV2005}
{\sc Adimurthi, S.~Mishra, and G.~D. Veerappa~Gowda}, {\em Optimal entropy
  solutions for conservation laws with discontinuous flux-functions}, J.
  Hyperbolic Differ. Equ., 2 (2005), pp.~783--837.

\bibitem{AMV2007}
\leavevmode\vrule height 2pt depth -1.6pt width 23pt, {\em Conservation law
  with the flux function discontinuous in the space variable—{II}:
  Convex--concave type fluxes and generalized entropy solutions}, J. Comput.
  Appl. Math., 203 (2007), pp.~310--344.

\bibitem{ACG2015}
{\sc A.~Aggarwal, R.~M. Colombo, and P.~Goatin}, {\em Nonlocal systems of
  conservation laws in several space dimensions}, SIAM J. Numer. Anal., 53
  (2015), pp.~963--983.

\bibitem{AG2016}
{\sc A.~Aggarwal and P.~Goatin}, {\em Crowd dynamics through non-local
  conservation laws}, Bull. Braz. Math. Soc. (N.S.), 47 (2016), pp.~37--50.

\bibitem{AHV2023}
{\sc A.~Aggarwal, H.~Holden, and G.~Vaidya}, {\em On the accuracy of the finite
  volume approximations to nonlocal conservation laws}, arXiv:2306.00142,
  (2023).

\bibitem{ASSV2020}
{\sc A.~Aggarwal, M.~R. Sahoo, A.~Sen, and G.~Vaidya}, {\em Solutions with
  concentration for conservation laws with discontinuous flux and its
  applications to numerical schemes for hyperbolic systems}, Stud. Appl. Math.,
  145 (2020), pp.~247--290.

\bibitem{AV2023}
{\sc A.~Aggarwal and G.~Vaidya}, {\em Convergence of finite volume
  approximations and well-posedness: {N}onlocal conservation laws with rough
  flux}, Preprint,  (2023).

\bibitem{AVV2021}
{\sc A.~Aggarwal, G.~Vaidya, and G.~D. Veerappa~Gowda}, {\em
  Positivity-preserving numerical scheme for hyperbolic systems with
  $\delta$-shock solutions and its convergence analysis}, Z. Angew. Math.
  Phys., 72 (2021), pp.~1--36.

\bibitem{ANT2007}
{\sc G.~Aletti, G.~Naldi, and G.~Toscani}, {\em First-order continuous models
  of opinion formation}, SIAM J. Appl. Math., 67 (2007), pp.~837--853.

\bibitem{ACT2015}
{\sc P.~Amorim, R.~M. Colombo, and A.~Teixeira}, {\em On the numerical
  integration of scalar nonlocal conservation laws}, ESAIM Math. Model. Numer.
  Anal., 49 (2015), pp.~19--37.

\bibitem{AKR2010}
{\sc B.~Andreianov, K.~H. Karlsen, and N.~H. Risebro}, {\em {$L^1$} stability
  for entropy solutions of nonlinear degenerate parabolic convection-diffusion
  equations with discontinuous coefficients.}, Netw. Heterog. Media, 5 (2010),
  pp.~617--633.

\bibitem{BR2020}
{\sc J.~Badwaik and A.~M. Ruf}, {\em Convergence rates of monotone schemes for
  conservation laws with discontinuous flux}, SIAM J. Numer. Anal., 58 (2020),
  pp.~607--629.

\bibitem{BBKT2011}
{\sc F.~Betancourt, R.~B{\"u}rger, K.~H. Karlsen, and E.~M. Tory}, {\em On
  nonlocal conservation laws modelling sedimentation}, Nonlinearity, 24 (2011),
  p.~855.

\bibitem{BG2016}
{\sc S.~Blandin and P.~Goatin}, {\em Well-posedness of a conservation law with
  non-local flux arising in traffic flow modeling}, Numer. Math., 132 (2016),
  pp.~217--241.

\bibitem{BP1998}
{\sc F.~Bouchut and B.~Perthame}, {\em Kruzkov’s estimates for scalar
  conservation laws revisited}, Trans. Amer. Math. Soc, 350 (1998),
  pp.~2847--2870.

\bibitem{AMW2021}
{\sc A.~Bressan, M.~T. Chiri, and W.~Shen}, {\em A posteriori error estimates
  for numerical solutions to hyperbolic conservation laws}, Arch. Ration. Mech.
  Anal., 241 (2021).

\bibitem{BS2020}
{\sc A.~Bressan and W.~Shen}, {\em A posteriori error estimates for
  self-similar solutions to the {E}uler equations}, Discrete Contin. Dyn.
  Syst., 41 (2021), pp.~113--130.

\bibitem{BKRT2004}
{\sc R.~B{\"u}rger, K.~H. Karlsen, N.~H. Risebro, and J.~D. Towers}, {\em
  Well-posedness in {${BV}_{t}$} and convergence of a difference scheme for
  continuous sedimentation in ideal clarifier-thickener units}, Numer. Math.,
  97 (2004), pp.~25--65.

\bibitem{CS2012}
{\sc C.~Canc{\`e}s and N.~Seguin}, {\em Error estimate for {G}odunov
  approximation of locally constrained conservation laws}, SIAM J. Numer.
  Anal., 50 (2012), pp.~3036--3060.

\bibitem{CG2023}
{\sc F.~A. Chiarello and P.~Goatin}, {\em {A non-local system modeling
  bi-directional traffic flows}}, SEMA SIMAI Springer Ser., 32 (2023).

\bibitem{CGL2012}
{\sc R.~M. Colombo, M.~Garavello, and M.~L{\'e}cureux-Mercier}, {\em A class of
  nonlocal models for pedestrian traffic}, Math. Mod. Met. Appl. Sci., 22
  (2012), pp.~1150023--34.

\bibitem{CHM2011}
{\sc R.~M. Colombo, M.~Herty, and M.~Mercier}, {\em Control of the continuity
  equation with a non local flow}, ESAIM Control Optim. Calc. Var., 17 (2011),
  pp.~353--379.

\bibitem{CL2011}
{\sc R.~M. Colombo and M.~L{\'e}cureux-Mercier}, {\em Nonlocal crowd dynamics
  models for several populations}, Acta Math. Sci. Ser. B, 32 (2011),
  pp.~177--196.

\bibitem{CM2015}
{\sc R.~M. Colombo and F.~Marcellini}, {\em Nonlocal systems of balance laws in
  several space dimensions with applications to laser technology}, J.
  Differential Equations, 259 (2015), pp.~6749--6773.

\bibitem{CMR2016}
{\sc R.~M. Colombo, F.~Marcellini, and E.~Rossi}, {\em Biological and
  industrial models motivating nonlocal conservation laws: A review of analytic
  and numerical results.}, Netw. Heterog. Media, 11 (2016), pp.~49--67.

\bibitem{CR2018}
{\sc R.~M. Colombo and E.~Rossi}, {\em Nonlocal conservation laws in bounded
  domains}, SIAM J. Math. Anal., 50 (2018), pp.~4041--4065.

\bibitem{FCV2023}
{\sc {Felisia Angela Chiarello}, H.~D. Contreras, and L.~M. Villada}, {\em
  Existence of entropy weak solutions for {1D} non-local traffic models with
  space discontinuous flux}, Journal of Engineering Mathematics,  (2023).

\bibitem{FGKP2022}
{\sc J.~Friedrich, S.~G{\"o}ttlich, A.~Keimer, and L.~Pflug}, {\em Conservation
  laws with nonlocal velocity--the singular limit problem}, arXiv:2210.12141,
  (2022).

\bibitem{GJT2020}
{\sc S.~S. Ghoshal, A.~Jana, and J.~D. Towers}, {\em Convergence of a {G}odunov
  scheme to an {A}udusse--{P}erthame adapted entropy solution for conservation
  laws with {BV} spatial flux}, Numer. Math., 146 (2020), pp.~629--659.

\bibitem{GTV2022a}
{\sc S.~S. Ghoshal, J.~D. Towers, and G.~Vaidya}, {\em Convergence of a
  {G}odunov scheme for degenerate conservation laws with {BV} spatial flux and
  a study of {P}anov-type fluxes}, J. Hyperbolic Differ. Equ., 19 (2022),
  pp.~365--390.

\bibitem{GTV2022}
\leavevmode\vrule height 2pt depth -1.6pt width 23pt, {\em A {G}odunov type
  scheme and error estimates for scalar conservation laws with {P}anov-type
  discontinuous flux}, Numer. Math., 151 (2022), pp.~601--625.

\bibitem{GHS+2014}
{\sc S.~G{\"o}ttlich, S.~Hoher, P.~Schindler, V.~Schleper, and A.~Verl}, {\em
  Modeling, simulation and validation of material flow on conveyor belts},
  Appl. Math. Model., 38 (2014), pp.~3295--3313.

\bibitem{GKLW2016}
{\sc M.~Gugat, A.~Keimer, G.~Leugering, and Z.~Wang}, {\em Analysis of a system
  of nonlocal conservation laws for multi-commodity flow on networks}, Netw.
  Heterog. Media, 10 (2016), pp.~749--785.

\bibitem{HKLR2010}
{\sc H.~Holden, K.~H. Karlsen, K.-A. Lie, and N.~H. Risebro}, {\em Splitting
  methods for partial differential equations with rough solutions: Analysis and
  MATLAB programs}, EMS Publishing House, Zürich, 2010.

\bibitem{HR2015}
{\sc H.~Holden and N.~H. Risebro}, {\em Front {T}racking for {H}yperbolic
  {C}onservation {L}aws}, Springer, second~ed., 2015.

\bibitem{KAR1994}
{\sc K.~H. Karlsen}, {\em On the accuracy of a numerical method for
  two-dimensional scalar conservation laws based on dimensional splitting and
  front tracking}, Preprint series 30, Department of Mathematics, University of
  Oslo,  (1994).

\bibitem{KR2001}
{\sc K.~H. Karlsen and N.~H. Risebro}, {\em Convergence of finite difference
  schemes for viscous and inviscid conservation laws with rough coefficients},
  ESAIM Math. Model. Numer. Anal., 35 (2001), pp.~239--269.

\bibitem{KP2021}
{\sc A.~Keimer and L.~Pflug}, {\em Discontinuous nonlocal conservation laws and
  related discontinuous {ODE}s---existence, uniqueness, stability and
  regularity}, arXiv:2110.10503,  (2021).

\bibitem{KR2013}
{\sc U.~Koley and N.~H. Risebro}, {\em Finite difference schemes for the
  symmetric {{Keyfitz--Kranzer}} system}, Z. Angew. Math. Phys., 64 (2013),
  pp.~1057--1085.

\bibitem{KUZ1976}
{\sc N.~N. Kuznetsov}, {\em Accuracy of some approximate methods for computing
  the weak solutions of a first-order quasi-linear equation}, USSR Comput.
  Math. Math. Phys., 16 (1976).

\bibitem{Pan2009}
{\sc E.~Y. Panov}, {\em On existence and uniqueness of entropy solutions to the
  {C}auchy problem for a conservation law with discontinuous flux}, J.
  Hyperbolic Differ. Equ., 6 (2009), pp.~525--548.

\bibitem{Per2007}
{\sc B.~Perthame}, {\em Transport {E}quations in {B}iology}, Frontiers in
  Mathematics, Birkh\"auser, 2007.

\bibitem{SAB1997}
{\sc F.~Sabac}, {\em The optimal convergence rate of monotone finite difference
  methods for hyperbolic conservation laws}, SIAM J. Numer. Anal., 34 (1997),
  pp.~2306--2318.

\bibitem{AS2012}
{\sc D.~A. Shen}, {\em An integro-differential conservation law arising in a
  model of granular flow}, J. Hyperbolic Differ. Equ., 9 (2012), pp.~105--131.

\bibitem{Tow2020}
{\sc J.~D. Towers}, {\em An existence result for conservation laws having {BV}
  spatial flux heterogeneities --- without concavity}, J. Differential
  Equations, 269 (2020), pp.~5754--5764.

\end{thebibliography}
	\end{document}